\documentclass[]{siamart190516}

\usepackage{amssymb}
\usepackage{amsmath}
\usepackage{algorithmic,algorithm}
\usepackage{url}

\newsiamremark{remark}{Remark}
\newsiamremark{example}{Example}
\newsiamremark{assumption}{Assumption}

\newcommand\norm[1]{\left\lVert#1\right\rVert}
\newcommand\scp[1]{\big\langle#1\big\rangle}

\makeatletter
\newcommand*\bigcdot{\mathpalette\bigcdot@{.5}}
\newcommand*\bigcdot@[2]{\mathbin{\vcenter{\hbox{\scalebox{#2}{$\m@th#1\bullet$}}}}}
\makeatother

\newcommand*\vp{\pmb y}
\newcommand*\vz{\pmb z}

\newcommand*\al{\pmb \alpha}

\newcommand*\eps{\pmb \epsilon}
\newcommand*\gam{\pmb \gamma}
\newcommand*\psipmb{\pmb \psi}

\newcommand*\vpmb{\pmb{v}}

\newcommand*\epmb{\pmb{e}}

\newcommand*\mbbc{\mathbb{C}}
\newcommand*\mbbn{\mathbb{N}}

\newcommand*\Ltwo{L^2(0,T;\mathbb{R}^M)}

\newcommand*\Bstar{B^{\star}}
\newcommand*\alstar{\pmb \alpha^{\star}}
\newcommand*\mustar{\mu^{\star}}
\newcommand*\kertxt{\textnormal{ker}}
\newcommand*\spantxt{\textnormal{span}}

\newcommand*\obsmat{\mathcal{O}_N(C,A)}

\usepackage{color}
\definecolor{Gabriele}{rgb}{0, 0, 1}

\newcommand{\TheTitle}{Analysis of a greedy reconstruction algorithm} 

\title{{\TheTitle}}
\author{S. Buchwald\thanks{Universit\"at Konstanz, Germany ({\tt simon.buchwald@uni-konstanz.de}).}
	\and
G. Ciaramella\thanks{Universit\"at Konstanz, Germany ({\tt gabriele.ciaramella@uni-konstanz.de}).}
	\and 
	J. Salomon\thanks{INRIA Paris, France ({\tt julien.salomon@inria.fr}).}
}

\begin{document}
	
\maketitle
	
\begin{abstract}
A novel and detailed convergence analysis is presented for
a greedy algorithm that was introduced in \cite{madaysalomon} for operator
reconstruction problems in the field of quantum mechanics.
This algorithm is based on an offline/online decomposition of the reconstruction process
and on an ansatz for the unknown operator obtained by an a priori chosen set of linearly
independent matrices.
The presented convergence analysis focuses on linear-quadratic (optimization) 
problems governed by linear differential systems and reveals the strong dependence of the performance of the greedy algorithm 
on the observability properties of the system and on the ansatz of the basis elements. Moreover, the analysis allows us to use a precise (and in some sense optimal) choice of basis elements for the linear case and led to the introduction of a new and more robust optimized greedy reconstruction algorithm. 
This optimized approach also applies to nonlinear Hamiltonian reconstruction problems,
and its efficiency is demonstrated by numerical experiments.
\end{abstract}
	
	

\section{Introduction}\label{sec:intro}

The identification of Hamiltonian operators plays a fundamental role in the fields
of quantum physics and quantum chemistry; see, e.g., \cite{Donovan2014,Geremia2001,Geremia2003,Geremia2000,Wang,Xue2019,Zhang2014,Zhou2012,Zhu1999,Rabitz0}
and references therein. Even though the overall literature about Hamiltonian identification problems 
is quite extensive, the mathematical contribution to this area is rather limited.
Important mathematical theoretical contributions can be found in~\cite{Bonnabel2009,Badouin2008}
and in~\cite{lebris:ham_id,Fu_2017}, where uniqueness results for quantum inverse problems are proved by exploiting controllability arguments.
Other techniques, based on the so-called Carleman's estimate, are used in~\cite{Badouin2008}
to deduce uniqueness results for inverse problems governed by Schr\"odinger-type equations 
in presence of discontinuous coefficients. 
Excluding these few theoretical results, the literature rather focuses on numerical algorithms.

The term Hamiltonian identification often refers to two distinct problems. 
On the one hand, it sometimes indicates the inverse problem associated with the identification of 
a Hamiltonian operator obtained by a numerical fitting of simulated and given experimental data. 
On the other hand, it occasionally refers to both the problem of designing experimental parameters 
(allowing an optimized production of experimental data) and the subsequent inverse identification problem.
In general, the design of experimental parameters includes the computation of control functions allowing an efficient numerical solving of the inverse problem. 

In the latter problem, the algorithms proposed in the literature often combine 
the computation of control functions with the production of new 
synthetic (simulated) data or experimental data. Mathematically, this framework has given rise 
to two different approaches. 
The first one~\cite{lebris:ham_id} consists in a procedure that alternately updates a (shrinking) set of admissible Hamiltonian operators and the trial control field used to generate new data. 
The second approach~\cite{madaysalomon} is based on a full offline/online decomposition and 
is inspired by the greedy strategy emerged in the field of approximation theory in the 2000s; 
see, e.g., \cite{eim2004} and references therein. Even though some mathematical investigations of 
the first approach can be found in the literature (see \cite{lebris:ham_id,Fu_2017}),
much less is known about the second strategy, for which only preliminary numerical results were 
presented in \cite{madaysalomon}.

The goal of the present work is to provide a first detailed convergence analysis of the
Hamiltonian reconstruction strategy defined in \cite{madaysalomon}. As a by-product,
this analysis allows us to introduce a new more efficient and robust numerical
reconstruction algorithm.

The numerical strategy presented in \cite{madaysalomon} is based on the ansatz that
the unknown operator can be written as a linear combination of a priori given
linearly independent matrices. The set of these matrices is denoted by $\mathcal{B}_\mu$.
The reconstruction process is then decomposed in offline phase and online phase. 
In the offline phase, a family of control functions is built iteratively in a greedy manner
in order to maximize the distinguishability of the system.
This phase exploits only the quantum model, without any use of laboratory information.  
The algorithm proposed in \cite{madaysalomon} for the offline phase, that we call in this paper
greedy reconstruction (GR) algorithm, consists of a sweep over the elements of $\mathcal{B}_\mu$.
At every iteration of the GR algorithm, one new element of $\mathcal{B}_\mu$ is considered
and a new control function is computed with the goal of splitting
the states generated by the new element and the ones already considered in the previous iterations.
The computed control functions are experimentally implemented in the online phase to produce
laboratory data.
These are in turn used to define and solve an identification 
inverse problem, aiming at fitting the numerical simulations with the corresponding
experimental data. 

In \cite{madaysalomon} the heuristic motivation for the offline phase 
is that this attempts to produce a set of control functions that make
the online identification problem uniquely solvable (and easier to be solved)
in a neighborhood of the true solution. 
Starting from this idea we develop a detailed convergence analysis
for linear problems (linear-quadratic in the least-squares sense).
Our analysis relates very clearly the iterations of the offline phase, and the corresponding
computed control functions, to the solvability of the online identification problem.
Moreover, the obtained theoretical results will reveal
the strong dependence of the performance of the greedy reconstruction algorithm 
on the observability properties of the system and on the ansatz of the basis elements used to reconstruct 
the unknown operator. These observations allow us to improve the GR algorithm and introduce
a new optimized greedy reconstruction (OGR) algorithm which shows a very robust behavior
not only for the linear-quadratic reconstruction problems, but also for nonlinear Hamiltonian
reconstruction problems.

The paper is organized as follows.
In Section \ref{sec:notation}, the notation used throughout this paper is fixed.
Section \ref{sec:bilinear_setting} describes the Hamiltonian reconstruction problem
and the original GR algorithm introduced in \cite{madaysalomon}.
The GR algorithm is then adapted to linear-quadratic problems in Section \ref{sec:LQ_problem}
and the corresponding convergence analysis is presented in Section \ref{sec:conv}.
In Section \ref{sec:impro}, we introduce some improvements of the GR algorithm that lead
to an optimized greedy reconstruction algorithm. The OGR algorithm is presented first for
linear-quadratic problems and then extend to nonlinear Hamiltonian reconstruction problems.
Within Section \ref{sec:impro}, results of numerical experiments are shown to
demonstrate the efficiency and the improved robustness of the new proposed algorithm.
Finally, we present our conclusions in Section \ref{sec:conclusions}.

\section{Notation}\label{sec:notation}
Consider a positive natural number $N$. 
We denote by $\scp{{\bf v},{\bf w}}:=\overline{{\bf v}}^\top{\bf w}$, for any
${\bf v},{\bf w}\in \mathbb{C}^N$ the usual complex scalar product on $\mbbc^N$, and by
$\norm{\, \cdot \,}_2$ the corresponding norm. Further, $| \, z \, |$ is the modulus
of a complex number $z$ and $i$ is the imaginary unit. 
The space of symmetric matrices in $\mathbb{R}^{N \times N}$ is
denoted by ${\rm Sym}(N)$. For any $A \in \mathbb{R}^{N \times N}$, 
$[A]_{j,k}$ denotes the $j,k$ (with $j,k \leq N$) entry of $A$ and
the notation $A_{[1:k,1:j]}$ indicates the upper left submatrix of $A$ of size $k \times j$,
namely $[A_{[1:k,1:j]}]_{\ell,m}:=[A]_{\ell,m}$ for $\ell=1,\dots,k$ and $m=1,\dots,j$.
Similarly, $A_{[1:k,j]}$ denotes the column vector in $\mathbb{R}^k$ corresponding to the
first $k$ elements of the column $j$ of $A$, namely  $[A_{[1:k,j]}]_{\ell} := [A]_{\ell,j}$
for $\ell=1,\dots,k$. 
Finally, the usual inner product of $L^2(0,T;\mbbc^N)$ is denoted by
$\scp{\cdot,\cdot}_{L^2}$, and $L^2:=L^2(0,T;\mathbb{R})$.

\section{Hamiltonian reconstruction and a greedy reconstruction algorithm}\label{sec:bilinear_setting}

Consider the finite-dimensional Schr\"odinger equation
\begin{equation}\label{eq:Schroedinger_equation}
\begin{split} 
i\dot{\psipmb}(t)&=[H+\epsilon(t)\mustar ]\psipmb(t),\quad t\in(0,T],\\
\psipmb(0)&=\psipmb_0,
\end{split}
\end{equation}
governing the time evolution of the state of a quantum system $\psipmb \in \mathbb{C}^N$, $N \in \mathbb{N}^+$.
The internal Hamiltonian $H$ is assumed to be known and the goal is to identify the unknown dipole moment 
operator $\mustar$ that couples the quantum system to a time-dependent external laser 
field $\epsilon\in L^2$, which acts as a control function on the system.
Both internal Hamiltonian $H$ and dipole operator $\mu^\star$ belong to ${\rm Sym}(N)$, and $\psipmb(t)$
lies in $\mathbb{C}^{N}$. 
The initial condition is $\psipmb_0\in\mathbb{C}^{N}$ which satisfies $\norm{\psipmb_0}_2=1$.

The true dipole operator $\mustar$ is unknown and assumed to lie in a space spanned by $K$ linearly independent
matrices $\mu_1,\ldots,\mu_K$, forming the set $\mathcal{B}_{\mu}=(\mu_j)_{j=1}^K \subset {\rm Sym}(N)$,
where $K\in\mbbn$ satisfies $1\leq K \leq {\rm dim \, Sym}(N)=\frac{N(N+1)}{2}$. 
Hence, we write $\mustar = \mu(\al^\star)$, with $\mu(\al) := \sum_{j=1}^K \al_j \mu_j$ for any $\al \in \mathbb{R}^K$.

To identify the true operator $\mustar$ one uses a set of control fields $(\epsilon^m)_{m=1}^K \subset L^2$
to perform $K$ laboratory experiments and obtain the experimental data 
$$\varphi(\mustar,\epsilon^m):= \scp{\psipmb_1,\psipmb_T(\mustar,\epsilon^m)}, \text{ for $m=1,\dots,K$}.$$ 
Here, $\psipmb_T(\mustar,\epsilon)$ denotes the solution to \eqref{eq:Schroedinger_equation} at time $T>0$, 
corresponding to the dipole operator $\mustar$ and a laser field $\epsilon$. 
The value $\psipmb_1\in\mathbb{C}^{N}$ is a fixed state with $\norm{\psipmb_1}_2=1$ and acts 
on a state of the quantum system as an observer operator. 
The measurements are assumed not to be affected by any type of noise.

Using the set of control fields $(\epsilon^m)_{m=1}^K$ and the corresponding
experimental data $(\varphi(\mustar,\epsilon^m))_{m=1}^K \subset \mathbb{C}$, one solves the nonlinear least-squares problem
\begin{align}
\min_{\al\in\mathbb{R}^K}\sum_{m=1}^{K}|\varphi(\mustar,\epsilon^m)-\varphi(\mu(\al),\epsilon^m)|^2,\label{eq:bilinear_main_problem}
\end{align}
where $\varphi(\mu(\al),\epsilon^m):= \scp{\psipmb_1,\psipmb_T(\mu(\al),\epsilon^m)}$,
with $\psipmb_T(\mu(\al),\epsilon^m)$ the solution to \eqref{eq:Schroedinger_equation} evaluated at time $T$ corresponding to
the dipole operator $\mu(\al)$ and the laser field $\epsilon^m$.
Clearly $\mu(\al^\star)$ is a global solution to \eqref{eq:bilinear_main_problem}.

If the control functions $(\epsilon^m)_{m=1}^K$ and the data $(\varphi(\mustar,\epsilon^m))_{m=1}^K$ are given,
problem \eqref{eq:bilinear_main_problem} is a standard parameter-identification inverse problem written in a 
minimization form.
The choice of the laser fields $(\epsilon^m)_{m=1}^K$ can affect significantly the properties
of \eqref{eq:bilinear_main_problem} and the corresponding solutions.
To design an optimized set of control functions, in particular with the goal of improving local convexity 
properties of \eqref{eq:bilinear_main_problem}, Maday and Salomon introduced in \cite{madaysalomon}
a numerical strategy which separates the reconstruction process of $\mustar$ in offline and online phases.
In the offline phase, a greedy reconstruction (GR) algorithm computes a set of optimized laser fields 
$(\epsilon^m)_{m=1}^K$ by exploiting only the quantum model \eqref{eq:bilinear_main_problem} 
and without using any laboratory data.
In the online phase, the computed control fields $(\epsilon^m)_{m=1}^K$ are used experimentally to produce the
laboratory data $\varphi(\mustar,\epsilon^m):= \scp{\psipmb_1,\psipmb_T(\mustar,\epsilon^m)}$ and to
solve the nonlinear problem \eqref{eq:bilinear_main_problem}.

While the online phase consists (mathematically) in solving a classical parameter-identification inverse problem,
the offline phase requires the GR algorithm introduced in \cite{madaysalomon}.
The ideal goal of this offline/online framework is to find a good approximation of the unknown 
operator for which the difference at time $T$  between observed experimental data and 
numerically computed data is the smallest for any control. 
In other words, one aims at finding a matrix $\mu$ that solves
\begin{align}
\min_{\mu\in\text{span}\,\mathcal{B}_{\mu}}\;\;\max_{\epsilon \in L^2} |\varphi(\mustar,\epsilon)-\varphi(\mu,\epsilon)|^2,\label{eq:bilinear_minmax}
\end{align}
or equivalently an $\al$ that solves
\begin{align}
\min_{\al\in\mathbb{R}^K}\;\;\max_{\epsilon \in L^2} |\varphi(\mu(\alstar),\epsilon)-\varphi(\mu(\al),\epsilon)|^2.\label{eq:bilinear_minmax_alpha}
\end{align}
Therefore, the goal of the GR algorithm is to generate a set of $K$ control functions such that 
a computed solution to \eqref{eq:bilinear_main_problem} is also a solution to \eqref{eq:bilinear_minmax}-\eqref{eq:bilinear_minmax_alpha}. 
To do so, the heuristic argument used in \cite{madaysalomon} is that the GR algorithm must attempt to distinguish numerical data for any two 
$\mu(\widetilde{\al}),\mu(\widehat{\al})\in {\rm span}\,\mathcal{B}_{\mu}$, $\mu(\widetilde{\al})\neq\mu(\widehat{\al})$, 
without performing any laboratory experiment. Following this idea, Maday and Salomon defined the GR algorithm as an iterative procedure
that performs a sweep over the linearly independent matrices $(\mu_k)_{k=1}^K$ and 
computes a new control field $\epsilon^{k+1}$ at each iteration.
Suppose that the control fields $\epsilon^1,\dots,\epsilon^k$ are already computed, the new control function $\epsilon^{k+1}$ is obtained by two sub-steps:
one first solves the identification problem
\begin{equation}\label{minsq}
\min_{\al_1,\dots,\al_k}\sum_{m=1}^k\Bigl\vert
\varphi(\sum_{j=1}^k\al_j\mu_j,\epsilon^m) - \varphi(\mu_{k+1},\epsilon^m)
\Bigr\vert^2,
\end{equation}
which gives the coefficients $\al_1^k,\dots,\al_k^k$,
and then computes the new field as
\begin{equation}\label{maxsq}
\epsilon^{k+1} \in \hbox{argmax}_{\epsilon \in L^2} \Bigl\vert \varphi(\mu_{k+1},\epsilon)-\varphi\Bigl(\sum_{j=1}^k\al_j^k\mu_j,\epsilon\Bigr)
\Bigr\vert^2.
\end{equation}
The step of solving Problem \eqref{minsq} is called \textit{fitting step}, since one attempts to compute a vector 
$\al^k:=[\al_1^k,\dots,\al_k^k]^\top$ that fits the quantities $\varphi(\sum_{j=1}^k\al_j^k\mu_j,\epsilon^m)$
and $\varphi(\mu_{k+1},\epsilon^m)$. In other words, the new basis element $\mu_{k+1}$ is considered and one identifies
an element $\mu^k(\al^k):=\sum_{j=1}^k \al_j^k \mu_j$ such that none of the already computed control functions
$\epsilon^1,\dots,\epsilon^k$ is capable of distinguishing the observations 
$\varphi(\mu^k(\al^k),\epsilon)$ and $\varphi(\mu_{k+1},\epsilon)$ (namely
$\varphi(\mu^k(\al^k),\epsilon^m) \neq \varphi(\mu_{k+1},\epsilon^m)$ for $m=1,\dots,k$).
The step of solving problem \eqref{maxsq} is called \textit{discriminatory step},
because one computes a control function $\epsilon^{k+1}$ that is capable of distinguishing (discriminating) 
$\varphi(\mu^k(\al^k)\mu_j,\epsilon^{k+1})$ from $\varphi(\mu_{k+1},\epsilon^{k+1})$.

The full GR algorithm is stated in Algorithm \ref{algo:bilinear_main}.\footnote{Notice that the 
initialization problem \eqref{eq:bilinear_initialization} is different from the one considered 
in \cite{madaysalomon}, which was stated anyway to be arbitrary. The reason for our choice 
is that (as we will see in the next sections) this slightly modified initialization 
problem \eqref{eq:bilinear_initialization} will be essential to obtain convergence.}
\begin{algorithm}[t]
	\caption{Greedy Reconstruction Algorithm}
	\begin{small}
	\begin{algorithmic}[1]\label{algo:bilinear_main} 
		\REQUIRE A set of $K$ linearly independent matrices $\mathcal{B}_{\mu}=(\mu_1,\ldots,\mu_K)$.
		\STATE Solve the initialization problem
		\begin{equation}\label{eq:bilinear_initialization}
		\max\limits_{\eps\in L^2} |\varphi(\mu_1,\epsilon) -\varphi(0,0)|^2.
		\end{equation}
		which gives the field $\epsilon^1$ and set $k=1$.
		\WHILE{ $k\leq K-1$ }
		\STATE \underline{Fitting step}: Find $(\al^{k}_j)_{j=1,\dots,k}$ that solve the problem
		\begin{equation}\label{eq:bilinear_fitting_step}
		\min\limits_{\al \in\mathbb{R}^k}\sum_{m=1}^{k}|\varphi(\mu_{k+1},\epsilon^m)-\varphi(\mu^k(\al),\epsilon^m) |^2.
		\end{equation}
		\STATE \underline{Discriminatory step}: Find $\epsilon^{k+1}$ that solves the problem
		\begin{equation}\label{eq:bilinear_discriminatory_step}
		\max\limits_{\epsilon\in L^2}|\varphi(\mu_{k+1},\epsilon)-\varphi(\mu^k(\al^k),\epsilon) |^2.
		\end{equation}
		\STATE Update $k \leftarrow k+1$.
		\ENDWHILE
	\end{algorithmic}
	\end{small}
\end{algorithm}
Notice how the algorithm is obtained by a sequence of minimization and maximization problems,
mimicking exactly the structure of the min-max 
problem \eqref{eq:bilinear_minmax}-\eqref{eq:bilinear_minmax_alpha}.

Notice also that, since the goal of the GR algorithm is to compute control functions that 
allow one to distinguish between the states of the system corresponding to any possible dipole matrix,
the algorithm implicitly attempts to compute control functions that make the online
identification problem \eqref{eq:bilinear_main_problem} locally strictly convex (hence uniquely solvable).
This is an important observation that we will use to begin our convergence analysis.

Let us conclude this section with a final remark about the laboratory measurements.
Throughout this paper, these are assumed to be not affected by any type of noise, even though 
noise is a significant factor that has to be dealt with; see \cite[Remark 1]{lebris:ham_id} and references therein.
However, the main goal of the present work is the numerical and convergence analysis of the computational framework 
and the GR algorithm introduced in \cite{madaysalomon}, where noisy effects in taking measurements are also neglected.

\section{Linear-quadratic reconstruction problems}\label{sec:LQ_problem}

Consider a state $\vp$ whose time evolution is governed by the (real) ordinary differential equation
\begin{equation}\label{eq: linear ODE}
\begin{split} 
\dot{\vp}(t)&=A\vp(t)+B^\star\eps(t),\quad t \in(0,T],\\
\vp(0)&=\vp_0,
\end{split}
\end{equation}
where $A\in\mathbb{R}^{N\times N}$ is a given matrix for $N\in\mathbb{N}^+$,
the initial condition is $\vp_0\in\mathbb{R}^N$,
and $\eps\in E_{ad}$ denotes a control function belonging to $E_{ad}$, a non-empty
and weakly compact subset of $\Ltwo$ (e.g., a closed, convex and bounded subset of $\Ltwo$).
The true control matrix $\Bstar \in \mathbb{R}^{N \times M}$,  for $M\in\mathbb{N}^+$, 
is unknown and assumed to lie in the space spanned 
by a set of linearly independent matrices 
$\mathcal{B}=\{B_1,\ldots,B_K\}\subset\mathbb{R}^{N\times M}$, $1\leq K\leq NM$,
and we write $\Bstar=\sum_{j=1}^{K}\al_j^\star B_j=:B(\al^\star)$.

As in the case of the Hamiltonian reconstruction problem, 
to identify the unknown matrix $\Bstar$ one can consider a set of control functions
$(\eps^m)_{m=1}^K \subset E_{ad}$ and use it experimentally to obtain the data
$C\vp_T(\Bstar,\eps^m)$, $m=1,\dots,K$. Here, $\vp_T(\Bstar,\eps)$ denotes the solution 
of \eqref{eq: linear ODE} at time $T$ and corresponding to a control function $\eps$ 
and to the control matrix $\Bstar$. Further, $C \in \mathbb{R}^{P\times N}$ is a given observer matrix. 

As in Section \ref{sec:bilinear_setting}, the reconstruction process is split into online
and offline phases. In the offline phase, the GR algorithm computes the control functions $(\eps^m)_{m=1}^K$.
These are then used in the online phase, in which the laboratory data 
$$C\vp_T(\Bstar,\eps^m), \; m=1,\dots,K$$
are obtained and the identification problem
\begin{align}
\min_{\al\in\mathbb{R}^K}\sum_{m=1}^{K} \norm{C\vp_T(\Bstar,\eps^m)-C\vp_T(B(\al),\eps^m) }_2^2 \label{eq: main problem}
\end{align}
is solved.

As for the Hamiltonian reconstruction problem, the ideal goal of the offline/online framework 
is to find a good approximation of the unknown operator for which the norm difference at time $T$ 
between observed experimental data and numerically computed data is the smallest for any control function. 
In other words, we wish to find a matrix $B$ that solves
\begin{align}
\min_{B\in{\rm span} \, \mathcal{B}}\max_{\eps\in E_{ad}} \norm{C\vp_T(\Bstar,\eps)-C\vp_T(B,\eps)}_2^2\label{eq:minmaxlinear}
\end{align}
or equivalently
\begin{align}
\min_{\al\in\mathbb{R}^K}\max_{\eps\in E_{ad}} \norm{C\vp_T(\Bstar,\eps)-C\vp_T(B(\al),\eps)}_2^2,\label{eq:minmaxlincoeff}
\end{align}
where $B(\al) := \sum_{j=1}^{K}\al_j B_j$.
The GR algorithm generates a set of $K$ controls that attempt to distinguish numerical data for any two 
$B(\widehat{\al}) \neq B(\widetilde{\al})$, without performing any laboratory experiment. 
The GR algorithm for linear-quadratic reconstruction problems is given
in Algorithm \ref{algo: main}.

\begin{algorithm}[t]
	\caption{Greedy Reconstruction Algorithm (linear-quadratic case)}
	\begin{small}
	\begin{algorithmic}[1]\label{algo: main} 
		\REQUIRE A set of $K$ linearly independent matrices $\mathcal{B}=(B_1,\ldots,B_{K})$.
		\STATE Solve the initialization problem
		\begin{equation}\label{eq: initialization}
		\max\limits_{\eps\in E_{ad}} \norm{C\vp_T(B_1,\eps) - \vp_T(0,0)}_2^2.
		\end{equation}
		which gives the field $\eps^1$ and set $k=1$.
		\WHILE{ $k\leq K-1$ }
		\STATE \underline{Fitting step}: Find $(\al^{k}_j)_{j=1,\dots,k}$ that solve the problem
		\begin{equation}\label{eq: fitting step}
		\min\limits_{\al \in\mathbb{R}^k}\sum_{m=1}^{k}\norm{C\vp_T(B_{k+1},\eps^m)-C\vp_T(B^k(\al),\eps^m) }_2^2.
		\end{equation}
		\STATE \underline{Discriminatory step}: Find $\eps^{k+1}$ that solves the problem
		\begin{equation}\label{eq: discriminatory step}
		\max\limits_{\eps\in E_{ad}}\norm{C\vp_T(B_{k+1},\eps)-C\vp_T(B^k(\al^k),\eps) }_2^2.
		\end{equation}
		\STATE Update $k \leftarrow k+1$.
		\ENDWHILE
	\end{algorithmic}
	\end{small}
\end{algorithm}

Since the convergence analysis performed in the next sections focuses on Algorithm \ref{algo: main},
we wish to explain it in more details. The idea is to generate controls that separate the observations 
of system \eqref{eq: linear ODE} at time $T$ for the different elements $B_1,\ldots,B_K$, 
making possible the identification of their respective coefficients $\al_1^{\star},\ldots,\al_K^{\star}$ 
when solving \eqref{eq: main problem}. The initialization is performed by solving the optimal control 
problem \eqref{eq: initialization}, which aims at maximizing the distance (at time $T$) between the observed state of 
the uncontrolled system (namely $\vp_T(0,0)$ corresponding to $\eps=0$) and the observed state of the system
\begin{equation*}
\begin{split}
\dot{\vp}(t)&=A\vp(t)+B_1\eps(t),\\
\vp(0)&=\vp_0.
\end{split}
\end{equation*}
The numerical solution of this maximization problem provides the first control function $\eps^1$.

Assume now that the control functions $\eps^1,\dots,\eps^k$ are computed. The new element
$\eps^{k+1}$ is obtained by performing a fitting step (namely solving problem \eqref{eq: fitting step})
and a discriminatory step (namely solving problem \eqref{eq: discriminatory step}).
In the fitting step, one compares the two systems
\begin{equation*}
\begin{cases} 
\dot{\vp}(t)=A\vp(t)+B_{k+1}\eps^m(t),\\
\vp(0)=\vp_0,
\end{cases}\qquad\quad 
\begin{cases} 
\dot{\vp}(t)=A\vp(t)+\biggl(\sum_{j=1}^{k}\al_jB_j\biggr)\eps^m(t),\\
\vp(0)=\vp_0,
\end{cases}
\end{equation*}
for $m\in\{1,\ldots,k \}$, and looks for an $\al\in\mathbb{R}^k$ for which their observed 
solutions  at time $T$ are as similar as possible (ideally the same, hence indistinguishable). 
We denote by $\al^k=[\al_1^k,\dots,\al_k^k]^\top$ the vector computed by solving \eqref{eq: fitting step}. 
This vector is used in the subsequent discriminatory step, 
which consists in solving the optimal control problem \eqref{eq: discriminatory step}. 
Here, we compute a control function $\eps^{k+1}$ that maximizes the distance (at time $T$) between 
the solutions of the two systems
\begin{equation*}
\begin{cases} 
\dot{\vp}(t)=A\vp(t)+B_{k+1}\eps(t),\\
\vp(0)=\vp_0,
\end{cases}\qquad\quad 
\begin{cases} 
\dot{\vp}(t)=A\vp(t)+\sum_{j=1}^{k}\al^k_j B_j \eps(t),\\
\vp(0)=\vp_0,
\end{cases}
\end{equation*}
where now $\al^k_j$ are fixed coefficients and the optimization variable is the control function $\eps$. 
Notice that this maximization problem is well posed, as we will discuss in 
Lemma \ref{lem:ALGinW} in Section \ref{sec:conv}.

We wish to remark again that, since the goal of the GR algorithm is to compute control functions 
that permit to distinguish between the states of the system corresponding to any possible control matrix,
the algorithm implicitly attempts to compute control functions that make the online
identification problem locally uniquely solvable.

With these preparations, we are ready to present our convergence analysis.

\section{Convergence Analysis}\label{sec:conv}

The convergence analysis presented in this section begins by recalling
that one of the goals of the GR algorithm is to compute a set control functions
that makes the online identification problem \eqref{eq: main problem} strictly convex
in a neighborhood of the solution $\al^\star$ (and hence locally uniquely solvable).
It is then natural to begin with problem \eqref{eq: main problem} and prove the following lemma,
which gives us an equivalent matrix-vector formulation.

\begin{lemma}[Online identification problem in matrix form]\label{lem: main problem terms of W}
Problem \eqref{eq: main problem} is equivalent to
\begin{align}\label{eq: P_W}
\min_{\al\in\mathbb{R}^K}\; \scp{\al^\star-\al,\widehat{W}(\al^\star-\al)},
\end{align}
where $\widehat{W}\in\mathbb{R}^{K\times K}$ is defined as
\begin{equation}\label{eq: What}
	\widehat{W}:=\sum_{m=1}^K W(\eps^m),
\end{equation}
with $W(\eps^m) \in\mathbb{R}^{K\times K}$ given by
\begin{equation}\label{eq: Wj}
	[W(\eps^m)]_{\ell,j}:=\scp{\gam_\ell(\eps^m),\gam_j(\eps^m)}, \text{ for $\ell,j=1,\dots,K$}, 
\end{equation}
and
\begin{equation}\label{eq: gamma_l,j}
	\gam_\ell(\eps^m):=\int_{0}^{T}Ce^{(T-s)A}B_\ell\eps^m(s)ds, \text{ for $m,\ell=1,\dots,K$}.
\end{equation}
\end{lemma}

\begin{proof}
Let us define 
\begin{equation}\label{eq: J_main}
J(\al):=\sum_{m=1}^K \norm{C\vp_T(\Bstar,\eps^m)-C\vp_T(B(\al),\eps^m) }_2^2,
\end{equation}
and notice that
$$\vp_T(\Bstar,\eps^m) = e^{TA} \vp_0 + \int_{0}^{T} e^{(T-s)A}\Big(\sum_{j=1}^K\al_j^\star B_j\Big)\eps^m(s)ds,$$
$$\vp_T(B(\al),\eps^m) = e^{TA} \vp_0 + \int_{0}^{T} e^{(T-s)A}\Big(\sum_{j=1}^K\al_j B_j \Big)\eps^m(s)ds.$$
The function $J(\al)$ can be written as
\begin{align*}
J(\al)
&=\sum_{m=1}^K\norm{\int_{0}^{T}Ce^{(T-s)A}\Big(\sum_{j=1}^K(\al_j^\star-\al_j)B_\ell\Big)\eps^m(s)ds}_2^2\\
&=\sum_{m=1}^K\sum_{\ell=1}^K\sum_{j=1}^K(\al_\ell^\star-\al_\ell)(\al_j^\star-\al_j)
\scp{\gam_\ell(\eps^m),\gam_{j}(\eps^m)},
\end{align*}
where the vectors $\gam_\ell(\eps^m)$ are defined in \eqref{eq: gamma_l,j}.
We can now write
\begin{align*}
J(\al)&=\sum_{\ell=1}^K\sum_{j=1}^K(\al_\ell^\star-\al_\ell)(\al_j^\star-\al_j)\sum_{m=1}^K\scp{\gam_\ell(\eps^m),\gam_{j}(\eps^m)}\\
&=\scp{\al^\star-\al,\sum_{m=1}^K W(\eps^m)(\al^\star-\al)}=\scp{\al^\star-\al,\widehat{W}(\al^\star-\al)},
\end{align*}
and the result follows.
\end{proof}

Notice that, the matrices $W(\eps^m)$ defined in \eqref{eq: Wj} can be written
as $W(\eps^m) = \Gamma(\eps^m)^\top \Gamma(\eps^m)$, 
where $\Gamma(\eps^m)=[\gam_1(\eps^m) \, \cdots \, \gam_K(\eps^m)]$.
Hence, $W(\eps^m)$ are symmetric and positive semi-definite. 
This guarantees that $\widehat{W}$ is also symmetric and positive semi-definite.
Therefore, problem \eqref{eq: P_W} is uniquely solved by $\al = \al^\star$ 
if and only if $\widehat{W}$ is positive definite, meaning that the GR algorithm
actually aims at computing a set
of control functions $(\eps^m)_{m=1}^K$ that makes $\widehat{W}$ positive definite.
We then need to study how the positivity of $\widehat{W}$ evolves during the
iteration of the algorithm.
To do so, the first step is to rewrite the three problems \eqref{eq: initialization},
\eqref{eq: fitting step} and \eqref{eq: discriminatory step} also in a matrix-vector form.

\begin{lemma}[The GR Algorithm \ref{algo: main} in matrix form]\label{lem:ALGinW}
Consider Algorithm \ref{algo: main}. It holds that:
\begin{itemize}\itemsep0em
\item The initialization problem \eqref{eq: initialization} is equivalent to
	\begin{equation}\label{eq: or. init. terms of W}
	\max_{\eps\in E_{ad}}\; [W(\eps)]_{1,1}.
	\end{equation}
	
\item The fitting-step problem \eqref{eq: fitting step} is equivalent to
	\begin{equation}\label{eq: fitting-step terms of W}
	\min_{\al\in\mathbb{R}^k}\; \scp{\al,\widehat{W}^k_{[1:k,1:k]}\al}-2\scp{\widehat{W}^k_{[1:k,k+1]},\al},
	\end{equation}
	where $\widehat{W}^k$ is the $k$-th partial sum of $\widehat{W}$, 
	that is $\widehat{W}^k=\sum_{m=1}^{k}W(\eps^m)$, and (recalling Section \ref{sec:notation})
	$\widehat{W}^k_{[1:k,1:k]}\in \mathbb{R}^{k \times k}$ denotes the $k\times k$
	upper-left block of $\widehat{W}^k$ and $\widehat{W}^k_{[1:k,k+1]} \in \mathbb{R}^k$
	is a column vector containing 
	the first $k$ components of the $k+1$-th column of $\widehat{W}^k$.
	
\item The discriminatory-step problem \eqref{eq: discriminatory step} is equivalent to
	\begin{equation}\label{eq: discriminatory-step terms of W}
	\max_{\eps\in E_{ad}}\; \scp{\vpmb,[W(\eps)]_{[1:k+1,1:k+1]}\vpmb},
	\end{equation}
	where $W(\eps)$ is defined in \eqref{eq: Wj} and $\vpmb:=[(\al^k)^\top,\;-1]^\top$.
\end{itemize}
Moreover, problems \eqref{eq: initialization}-\eqref{eq: or. init. terms of W},
\eqref{eq: fitting step}-\eqref{eq: fitting-step terms of W}, 
and \eqref{eq: discriminatory step}-\eqref{eq: discriminatory-step terms of W}
are well posed.
\end{lemma}

\begin{proof}
The equivalences between \eqref{eq: initialization},
\eqref{eq: fitting step}, \eqref{eq: discriminatory step}
and \eqref{eq: or. init. terms of W}, \eqref{eq: fitting-step terms of W}, 
and \eqref{eq: discriminatory-step terms of W}, respectively, can be proved by similar calculations
to the one used in the proof of Lemma \ref{lem: main problem terms of W}.
We omit them for brevity.

Problem \eqref{eq: fitting step}-\eqref{eq: fitting-step terms of W} is a quadratic
minimization problem with quadratic function bounded from below by zero. 
Hence the existence of a minimizer follows.

Problems \eqref{eq: initialization}-\eqref{eq: or. init. terms of W}
and \eqref{eq: discriminatory step}-\eqref{eq: discriminatory-step terms of W}
are two classical optimal control problems. Since the admissible set $E_{ad}$
is a weakly compact subset of $\Ltwo$, the existence of a maximizer
follows by standard arguments based on maximizing sequences and weak compactness;
see, e.g., \cite{LibroQuantum} and references therein.
\end{proof}

Using the matrix representation given in Lemma \ref{lem:ALGinW}, we can now sketch
the mathematical meaning of the iterations of the GR algorithm.
Assume that at the $k$-th iteration the submatrix $\widehat{W}^k_{[1:k,1:k]}$ is positive definite,
but $\widehat{W}^k_{[1:k+1,1:k+1]}$ has a non-trivial (one-dimensional) kernel.
The GR algorithm first tries to identify (by solving problem \eqref{eq: fitting-step terms of W}) the kernel of $\widehat{W}^k_{[1:k+1,1:k+1]}$,
and then attempts to compute (by solving problem \eqref{eq: discriminatory-step terms of W}) a new control function $\eps^{k+1}$ such that the matrix
$W_{[1:k+1,1:k+1]}(\eps^{k+1})$ is positive on the kernel $\widehat{W}^k_{[1:k+1,1:k+1]}$.
If these happen, then the new updated matrix 
$\widehat{W}^{k+1}=\widehat{W}^k+W(\eps^{k+1})$ has
a positive definite upper-left block $\widehat{W}^{k+1}_{[1:k+1,1:k+1]}$.
Moreover, if these two steps hold for any $k$, then the convergence follows since
after the $K-1$-th iteration the matrix $\widehat{W}=\widehat{W}^K$ results to be
positive definite.
Hence, two questions clearly arise:
\begin{enumerate}\itemsep0em
\item Does the fitting step of the algorithm always compute
the non-trivial kernel of $\widehat{W}^k_{[1:k+1,1:k+1]}$ (in case it is truly non trivial)?
\item Does the discriminatory step of the algorithm always compute 
a control function $\eps^{k+1}$ that makes $\widehat{W}^{k+1}_{[1:k+1,1:k+1]}$ positive definite?
\end{enumerate}

The first question can be answered with the help of the following technical lemma.

\begin{lemma}[On the kernel of symmetric positive semi-definite matrices]\label{lem: if A pos def then also next bigger A}
	Consider a symmetric, positive semi-definite matrix $\widetilde{G}\in\mathbb{R}^{n\times n}$ of the form
	\begin{equation*}
	\widetilde{G}=\begin{bmatrix}
	G&\pmb b\\
	\pmb b^\top&c
	\end{bmatrix},
	\end{equation*}
	where $G\in\mathbb{R}^{(n-1)\times(n-1)}$ is symmetric and positive definite 
	and $\pmb b\in\mathbb{R}^{n-1}$ and $c\in\mathbb{R}$ are such that the kernel of $\widetilde{G}$ 
	is non-trivial.	Then
	\begin{equation*}
	{\rm ker}(\widetilde{G})={\rm span}\bigg\{\begin{bmatrix}
	G^{-1}\pmb{b}\\-1
	\end{bmatrix} \bigg\}.
	\end{equation*}
\end{lemma}
\begin{proof}
	Since the kernel of $\widetilde{G}$ is non-trivial, there exists a non-zero vector 
	\linebreak $\pmb{u}=\begin{bmatrix}\pmb{v}\\d\end{bmatrix}\in\mathbb{R}^n\setminus \{0\}$ 
	(with $\pmb{v} \in \mathbb{R}^{n-1}$ and $d \in \mathbb{R}$) such that $\widetilde{G}\pmb{u}=0$. 
	Moreover, since $G$ is positive definite, the kernel of $\widetilde{G}$ must be one-dimensional 
	and equal to the span of $\{\pmb{u}\}$. Using the structure of $\pmb{u}$, we write
	$\widetilde{G} \pmb{u} = 0$ as
	\begin{equation}\label{eq:prooffff}
	\begin{cases}
	G\pmb{v}+d\;\pmb{b}=0,\\
	\pmb{b}^\top\pmb{v}+dc=0,
	\end{cases}
	\quad\overset{G \text{ invertible}}{\Longleftrightarrow}\quad
	\begin{cases}
	\pmb{v}=-dG^{-1}\pmb{b},\\
	-d\;\pmb{b}^\top G^{-1}\pmb{b}+dc=0.
	\end{cases}
	\end{equation}
	Now, suppose that $d=0$. This implies that $\pmb{v}=-dG^{-1}\pmb{b}=0$, which in turn implies
	that $\pmb{u}=0$. However, this is a contradiction to the fact that $\pmb{u} \neq 0$. 
	Hence $d \neq 0$. The result follows by the right equations in \eqref{eq:prooffff}
	(divided by $-d$).
\end{proof}

Recalling the equivalent form \eqref{eq: fitting-step terms of W} of the fitting-step problem
\eqref{eq: fitting step}, one clearly see that, if $\widehat{W}^k_{[1:k,1:k]}$ is positive definite,
then the unique solution to \eqref{eq: fitting-step terms of W} is given by $\al^k~=~(\widehat{W}^k_{[1:k,1:k]})^{-1}\widehat{W}^k_{[1:k,k+1]}$.
On the other hand, if we set
\begin{equation*}
\widetilde{G}=\widehat{W}^k_{[1:k+1,1:k+1]},\;G=\widehat{W}^k_{[1:k,1:k]},\; \pmb{b}=\widehat{W}^k_{[1:k,k+1]},\;c=\widehat{W}^k_{[k+1,k+1]},
\end{equation*}
then Lemma \ref{lem: if A pos def then also next bigger A} guarantees that the vector
$\vpmb:=[(\al^k)^\top,\;-1]^\top$ spans the kernel of $\widehat{W}^k_{[1:k+1,1:k+1]}$,
if this is non-trivial. Therefore, we have 
\begin{align*}
\kertxt(\widehat{W}^k_{[1:k+1,1:k+1]})&=\spantxt\Big\{\begin{bmatrix}
(\widehat{W}^k_{[1:k,1:k]})^{-1}\widehat{W}^k_{[1:k,k+1]}\\-1
\end{bmatrix} \Big\}=\spantxt\Big\{\vpmb:=\begin{bmatrix}
\al^k\\-1
\end{bmatrix} \Big\}.
\end{align*}
This means that, if $\widehat{W}^k_{[1:k+1,1:k+1]}$ has a rank defect, then the GR algorithm
finds it by the splitting step.

The answer to the second question posed above is more complicated.
In order to formulate it properly, we need to recall the definition of observability
of an input/output dynamical system of the form
\begin{equation}\label{eq:IOSyst}
\begin{split} 
\dot{\vp}(t)&=A\vp(t)+B\eps(t),\quad \vp(0)=\vp_0, \\
\vz(t)&=C\vp(t),
\end{split}
\end{equation}
with $A\in\mathbb{R}^{N \times N}$, $B\in\mathbb{R}^{N \times M}$, $C\in\mathbb{R}^{P \times N}$;
see, e.g., \cite{Sontag1998}.

\begin{definition}[Observable input-output linear systems]
The input-output linear system \eqref{eq:IOSyst} is said to be observable if the initial state
$\vp(0)=\vp_0$ can be uniquely determined from input/output measurements.
Equivalently, \eqref{eq:IOSyst} is observable if and only if the observability matrix
\begin{equation}
\mathcal{O}_N(C,A) := \begin{bmatrix}
	C\\ CA\\ \vdots\\CA^{N-1}
	\end{bmatrix}
\end{equation}
has full column rank.
\end{definition}
Notice that the matrix $B$ does not affect the observability of system \eqref{eq:IOSyst}.

We now assume that the system is observable, namely that $\textnormal{rank}\;\mathcal{O}_N(C,A) =N$,
and we show that this is a sufficient condition for the GR algorithm to make the matrix
$\widehat{W}$ positive definite.
To do so, we first prove the following lemma regarding the discriminatory step.
Notice that the proof of this result is inspired by classical Kalmann controllability
theory; see, e.g., \cite{coron2007control}.

\begin{lemma}[Discriminatory-step problem for fully observable systems]\label{prop: discr positivity}
	Assume that the matrices $A \in \mathbb{R}^{N \times N}$ and $C \in \mathbb{R}^{P \times N}$ are such that
	$\textnormal{rank }\mathcal{O}_N(C,A)=N$.
	Let $\widehat{W}^k_{[1:k,1:k]}$ be positive definite, $\al^k$ the solution to the fitting-step 
	problem \eqref{eq: fitting step}, and $\vpmb =[(\al^k)^\top,-1]^\top$. 
	Then any solution $\eps^{k+1}$ of the discriminatory-step 	problem 
	\eqref{eq: discriminatory step} satisfies
	\begin{equation*}\label{eq: pos.discr.}
	\scp{\vpmb, W_{[1:k+1,1:k+1]}(\eps^{k+1})\vpmb}
	=\norm{\int_{0}^{T}Ce^{(T-s)A}\Big(B_{k+1}-\sum_{j=1}^{k}\al^k_j B_j \Big)\eps^{k+1}(s)ds}_2^2>0,
	\end{equation*}
	for $k=0,1,\dots,K-1$.
\end{lemma}

\begin{proof}
	Let us define $\widetilde{B}:=B_{k+1}-\sum_{j=1}^{k}\al^k_j B_j$.
	Since the matrices $B_1,\ldots,B_{k+1}$ are assumed to be linearly independent, $\widetilde{B}$ is non-zero.
	
	Now, we consider an arbitrary $\delta\in (0,T)$ and define a control function
	$\widetilde{\eps}\in E_{ad}$ as
	\begin{equation*}
	\widetilde{\eps}(s) :=\begin{cases}
	0,& 0\leq s< \delta,\\
	\epmb_i,& \delta\leq s\leq T,
	\end{cases}
	\end{equation*}
	where $\epmb_i\in\mathbb{R}^M$ is the $i$-th canonical vector for some index $1\leq i\leq M$.
	Further, we denote by $\widetilde{\pmb{b}}_i$ the $i$-th column of $\widetilde{B}$. 
	Since $\widetilde{B}$ is non-zero, we can choose the index $i$ such that $\widetilde{\pmb{b}}_i \neq 0$.
	Now, we compute
	\begin{align*}
	\int_{0}^{T}Ce^{(T-s)A}\widetilde{B}\widetilde{\eps}(s)ds
	&=\int_{\delta}^{T}Ce^{(T-s)A}\widetilde{\pmb{b}}_i ds
	=\int_{\delta}^{T}C\Big[\sum_{j=0}^{\infty}\frac{(T-s)^jA^j}{j!}\Big] \widetilde{\pmb{b}}_i ds\\
	&\overset{(\star)}{=}\Big[\sum_{j=0}^{\infty}\int_{\delta}^{T}\frac{(T-s)^j}{j!}ds\;CA^j\Big] \widetilde{\pmb{b}}_i
	=\Big[\sum_{j=0}^{\infty}\frac{(T-\delta)^{j+1}}{(j+1)!}CA^j\Big] \widetilde{\pmb{b}}_i\\
	&=\sum_{j=0}^{\infty}\beta_j(\delta) CA^j \widetilde{\pmb{b}}_i,
	\end{align*}
	where $\beta_j(\delta):=\frac{(T-\delta)^{j+1}}{(j+1)!}$ and 
	we used the dominated convergence theorem (see, e.g., \cite[Theorem 1.34]{rudin}) to
	interchange integral and infinite sum and obtain the equality $(\star)$. 
	Since the observability matrix $\mathcal{O}_N(C,A)$ has full rank 
	and $\widetilde{\pmb{b}}_i\neq 0$, there exists an index 
	$0\leq j\leq N-1$ such that $CA^j \widetilde{\pmb{b}}_i \neq 0$.
	Hence, $f(\delta):= \sum_{j=0}^{\infty}\beta_j(\delta) CA^j \widetilde{\pmb{b}}_i$
	is an analytic function for $\delta \in (0,T)$ and such that $f \neq 0$.\footnote{To see it, recall that
	$\beta_j(\delta)=\frac{(T-\delta)^{j+1}}{(j+1)!}$,
	consider a function $g(x)=\sum_{j=0}^{\infty}\frac{x^{j+1}}{(j+1)!} \gamma_j$, and assume that there exists
	at least one integer $k$ such that $\gamma_k \neq 0$. Now, if we pick the minimum integer $\widehat{k}$ such that
	$\gamma_{\widehat{k}} \neq 0$, we have that 
	$g(x) = \frac{x^{\widehat{k} +1}}{(\widehat{k} +1)!} \gamma_{\widehat{k}} + \sum_{j=\widehat{k}+1}^{\infty}\frac{x^{j+1}}{(j+1)!} \gamma_j$.
	For $x \rightarrow 0$, the first term behaves as $O(x^{\widehat{k} +1})$, while the second term as 
	$O(x^{\widehat{k} +2})$. Hence, there exists a point $y>0$ such that $g(y) \neq 0$.
	}
	We also know that (non-constant) analytic functions have isolated roots; see, e.g., \cite[Theorem 10.18]{rudin}.
	Therefore we can find a $\delta\in(0,T)$ such that 
	$\sum_{j=0}^{\infty}\beta_j(\delta) CA^j \widetilde{\pmb{b}}_i\neq 0$,
	and obtain the existence of an $\widetilde{\eps}\in E_{ad}$ such that 
	\begin{equation*}
	\int_{0}^{T}Ce^{(T-s)A}\widetilde{B}\widetilde{\eps}(s)ds\neq 0.
	\end{equation*}
	This implies that
	\begin{equation*}
	\begin{split}
	\scp{\vpmb, W_{[1:k+1,1:k+1]}(\eps^{k+1})\vpmb}
	&=\norm{\int_{0}^{T}Ce^{(T-s)A}\Big(B_{k+1}-\sum_{\ell=1}^{k}\al^k_\ell B_\ell \Big)\eps^{k+1}(s)ds}_2^2	\\
	&\geq\norm{\int_{0}^{T}Ce^{(T-s)A}\Big(B_{k+1}-\sum_{\ell=1}^{k}\al^k_\ell B_\ell \Big)\widetilde{\eps}(s)ds}_2^2\\
	&=\norm{\int_{0}^{T}Ce^{(T-s)A}\widetilde{B}\widetilde{\eps}(s)ds}_2^2>0,
	\end{split}
	\end{equation*}
	where we have used that $\eps^{k+1}$ is a maximizer for 
	problem \eqref{eq: discriminatory step}.
\end{proof}

Now we can prove our first main convergence result.

\begin{theorem}[Convergence of the GR algorithm for fully observable systems]\label{thm: overall main conv}
	Assume that the matrices $A \in \mathbb{R}^{N \times N}$ and $C \in \mathbb{R}^{P \times N}$ are such that
	$\textnormal{rank }\mathcal{O}_N(C,A)=N$.
	Let $K\in \{1,\dots,MN\}$ be arbitrary and let
	$\{\eps^1,\ldots,\eps^K \}\subset E_{ad}$ be a family of controls generated by the 
	GR Algorithm \ref{algo: main}. Then the matrix $\widehat{W}$ defined in \eqref{eq: What}
	is positive definite and online identification problem \eqref{eq: main problem} 
	is uniquely solvable by $\al=\al^\star$.
\end{theorem}

\begin{proof}
	By Lemma \ref{lem: main problem terms of W} it is sufficient to show that the matrix $\widehat{W}$ 
	corresponding to the controls $\eps^1,\ldots,\eps^K$ generated by the algorithm is positive definite. 
	The proof of this claim proceeds by induction.

	Lemma \ref{prop: discr positivity} guarantees that there exists an $\eps^1$ such that $[W(\eps^1)]_{1,1}>0$.
	Now, we assume that $\widehat{W}^k_{[1:k,1:k]}$ is positive definite, 
	and we show that $\widehat{W}^{k+1}_{[1:k+1,1:k+1]}$ is positive definite as well.
	
	If $\widehat{W}^{k}_{[1:k+1,1:k+1]}$ is positive definite, then 
	$$\widehat{W}^{k+1}_{[1:k+1,1:k+1]} = \widehat{W}^{k}_{[1:k+1,1:k+1]} + W(\eps^k)_{[1:k+1,1:k+1]}$$ 
	is positive definite as well, since $W(\eps^k)_{[1:k+1,1:k+1]}$ is positive semi-definite.
	
	Assume now that the submatrix $\widehat{W}^k_{[1:k+1,1:k+1]}$ has a non-trivial kernel. 	
	Since $\widehat{W}^k_{[1:k,1:k]}$ is positive definite (induction hypothesis), 
	problem \eqref{eq: fitting-step terms of W} is uniquely solvable with solution $\al^k$. 
	Then, by Lemma \ref{lem: if A pos def then also next bigger A} the (one-dimensional) kernel 
	of $\widehat{W}^k_{[1:k+1,1:k+1]}$ is the span of the the vector $\vpmb =[(\al^k)^\top,\;-1]^\top$.
	Finally, using Lemma \ref{prop: discr positivity} we obtain that the solution $\eps^{k+1}$ to the discriminatory-step 
	problem satisfies
	\begin{equation*}
	0<\scp{\vpmb,[W(\eps^{k+1})]_{[1:k+1,1:k+1]}\vpmb}.
	\end{equation*}
	Hence, the matrix $[W(\eps^{k+1})]_{[1:k+1,1:k+1]}$ is positive definite 
	on the span of $\vpmb$. Therefore 
	$\widehat{W}^{k+1}_{[1:k+1,1:k+1]}=\widehat{W}^k_{[1:k+1,1:k+1]}+[W(\eps^{k+1})]_{[1:k+1,1:k+1]}$
	is positive definite, and the claim follows.
\end{proof}

\begin{remark}[Uniqueness of solution of the min-max problem \eqref{eq:minmaxlincoeff}]\label{rem:minmaxequivmainconv}
	Under the assumption that the system is fully observable, 
	the min-max problem \eqref{eq:minmaxlincoeff} is also uniquely solvable with $\al=\al^\star$. 
	To see this, we first note that \eqref{eq:minmaxlincoeff} can be written in terms of $W(\eps)$:
	\begin{align*}
	\norm{C\vp_T(\Bstar,\eps)-C\vp_T(B(\al),\eps)}^2&=\norm{\int_{0}^{T}Ce^{(T-s)A}\Big(\sum_{j=1}^K (\al_j-\alstar_j)B_j\Big)\eps(s)ds}_2^2\\
	&=\scp{(\al-\alstar),W(\eps)(\al-\alstar)}.
	\end{align*}
	Now, similarly as in the proof of Lemma \ref{prop: discr positivity} and using the full observability of the system, 
	one can show that for any $\widehat{\al}\in\mathbb{R}^{NM}$ with $\widehat{\al}\neq\alstar$ there exists a 
	control $\eps(\widehat{\al})$ such that 
	\begin{equation*}
	\scp{(\widehat{\al}-\alstar),W(\eps(\widehat{\al}))(\widehat{\al}-\alstar)}>0.
	\end{equation*}
	Therefore the unique solution to \eqref{eq:minmaxlincoeff} is $\al=\al^\star$.
\end{remark}

Notice that, Theorem \ref{thm: overall main conv} does not require any particular assumption
on the matrices $B_1,\dots,B_K$, which can be arbitrarily chosen with the only constraint to be linearly
independent. Moreover, the number $K \in \{ 1 , \dots , MN\}$ can be fixed arbitrarily and the GR algorithm
will compute control functions that permit the exact reconstruction of the coefficients of the linear
combination of the first $K$ components of $B^\star$ in a basis $\{ B_1,\dots,B_{MN} \}$.
To be more precise, if the unknown $B^\star$ belongs to the span of $K$ the linearly independent
matrices $B_1,\dots,B_K$ used by the algorithm, then, using the control functions generated by the GR
algorithm, the unknown $B^\star$ can be fully reconstructed.
If $B^\star$ lies in the span of $\widetilde{K} \in \{K+1,K+2,\dots,MN\}$ linearly independent 
matrices $B_1,\dots,B_{\widetilde{K}}$, but only the first $K$ of these are used by the algorithm 
(and in the online identification problem), then one reconstructs exactly the $K$ coefficients 
corresponding to the first elements $B_1,\dots,B_K$.
Furthermore, the ordering of the $K$ considered matrices does not affect the convergence result
of Theorem \ref{thm: overall main conv}.

These observations are no longer true if the system is non-fully observable, that is
${\rm rank}\,\mathcal{O}_N(C,A)=\mathcal{R}<N$.
In this case, the choice of the linearly independent matrices $B_1,\dots,B_K$ and their ordering become
crucial for the algorithm. In particular, we are going to show that the method can recover at most 
$K=\mathcal{R}M$ components of the unknown vector $\al^\star$, if appropriate matrices $B_1,\dots,B_K$ are
chosen. Moreover, we will see that an inappropriate choice of matrices $B_1,\dots,B_K$
can lead to completely wrong results with an arbitrary error.

For our analysis, we begin by choosing a set of $K=NM$ matrices by exploiting the kernel of the observability
matrix. In particular, recalling that ${\rm rank}\,\mathcal{O}_N(C,A)=\mathcal{R}<N$, the rank-nullity
theorem allows us to consider a basis $\{\vpmb_j\}_{j=1}^{N}\subset \mathbb{R}^N$ of $\mathbb{R}^N$, 
such that
\begin{align}
\vpmb_j\notin \textnormal{ker}\;\mathcal{O}_N(C,A), &\qquad\; j=1,\ldots,\mathcal{R}, \label{eq:vj_1} \\
\vpmb_j\in \textnormal{ker}\;\mathcal{O}_N(C,A), &\qquad j=\mathcal{R}+1,\ldots,N, \label{eq:vj_2}
\end{align}
where clearly $\textnormal{span}\{\vpmb_j\}_{j=\mathcal{R}+1}^{N}= \textnormal{ker}\;\mathcal{O}_N(C,A)$.
We now define a basis $\{B^\mathcal{O}_k\}_{k=1}^{NM}$ of $\mathbb{R}^{N\times M}$ as
\begin{equation}\label{eq:our_basis}
\begin{split}
B^\mathcal{O}_1 &= \vpmb_1\epmb_1^\top, \; B^\mathcal{O}_2 = \vpmb_1\epmb_2^\top, \; \cdots, \; B^\mathcal{O}_M = \vpmb_1\epmb_M^\top, \\
B^\mathcal{O}_{M+1} &= \vpmb_2\epmb_1^\top, \; B^\mathcal{O}_{M+2} = \vpmb_2\epmb_2^\top, \; \cdots, \; B^\mathcal{O}_{2M} = \vpmb_2\epmb_M^\top, \\
&\vdots \qquad \qquad \qquad \vdots \qquad \qquad  \qquad \vdots\\
B^\mathcal{O}_{(N-1)M+1} &= \vpmb_N\epmb_1^\top, \; B^\mathcal{O}_{(N-1)M+2} = \vpmb_N\epmb_2^\top, \; \cdots, \; B^\mathcal{O}_{NM} = \vpmb_N\epmb_M^\top,
\end{split}
\end{equation}
where $\epmb_\ell\in\mathbb{R}^M$, for $\ell=1,\dots,M$, are the canonical vectors in $\mathbb{R}^M$.
Notice that, since the vectors $(\vpmb_j)_{j=1}^N$ are linearly independent, the set $\{B^\mathcal{O}_k\}_{k=1}^{NM}$ 
is clearly a basis of $\mathbb{R}^{N\times M}$. 

From a computational point of view, the vectors $\vpmb_j$ can be obtained by a singular value decomposition (SVD) 
of the observability matrix $\obsmat = U\Sigma V^\top$, where the columns of $V$ form a basis of $\mathbb{R}^N$
and the last $N-\mathcal{R}$ columns of $V$ span the kernel of $\obsmat$; see, e.g., \cite[Theorem 5.2]{trefethen97}.
Therefore, one can set $\vpmb_j = V_{[:,j]}$, $j=1,\ldots,N$.

Our first result for non-fully observable systems says that, if the basis $\{B^\mathcal{O}_k\}_{k=1}^{NM}$
is considered, then we can reduce the reconstruction of 
$B^\star=\sum_{j=1}^{MN} \al_j^\star B^\mathcal{O}_j$ only to the first $\mathcal{R}M$ coefficients
$\al_1,\dots,\al_{\mathcal{R}M}$. This is proved in the next lemma, where we use the notation
\begin{equation*}
B_{\mathcal{R}}(\al^\star):=\sum_{j=1}^{\mathcal{R}M}\al_j^\star B^\mathcal{O}_j.
\end{equation*}

\begin{lemma}[Online identification problem for non-fully observable systems]\label{lem:mainproblemlesscoefficients}
	Consider the basis $\{B^\mathcal{O}_k\}_{k=1}^{NM}$ constructed as in \eqref{eq:our_basis}
	(with vectors $\vpmb_j$, $j=1,\dots,N$, as in \eqref{eq:vj_1}-\eqref{eq:vj_2}).
	The online least-squares problem \eqref{eq: main problem} (with $K=MN$) is equivalent to
	\begin{align}\label{eq: impr main problem}
	\min_{\al\in\mathbb{R}^{\mathcal{R}M}}\; \sum_{m=1}^{NM}\norm{C\vp_T(B^\star,\eps^m)-C\vp_T(B_{\mathcal{R}}(\al),\eps^m) }_2^2.
	\end{align}
\end{lemma}

\begin{proof}
	Notice that, for any $\ell \in \{1,2,\dots, NM\}$ and $s \in [0,T]$, there exist
	$N$ functions $\widetilde{\beta}_j$ such that
	\begin{align*}
	Ce^{(T-s)A}B^\mathcal{O}_\ell&=C\sum_{j=0}^{\infty}\frac{(T-s)^j}{j!}A^j B^\mathcal{O}_\ell 
	\overset{(\star)}{=}C\Big[\sum_{j=0}^{N-1}\widetilde{\beta}_j(s) A^j\Big] B^\mathcal{O}_\ell\\
	&=\Big[\widetilde{\beta}_0(s)I_N,\widetilde{\beta}_1(s)I_N,\ldots,\widetilde{\beta}_{N-1}(s)I_N \Big]\mathcal{O}_N(C,A)B^\mathcal{O}_\ell,
	\end{align*}
	where we have used the Cayley-Hamilton theorem (see, e.g., \cite[p.109]{horn2012}) to obtain the equality $(\star)$. 
	If $\ell\in\{\mathcal{R}M+1,\ldots,NM \}$, then $B^\mathcal{O}_\ell=\vpmb_j\epmb_i^\top$ with $j\geq \mathcal{R}+1$, 
	hence $\vpmb_j\in \textnormal{ker}\;\mathcal{O}_N(C,A)$ and therefore
	\begin{equation*}
	\mathcal{O}_N(C,A)B^\mathcal{O}_\ell = \underbrace{\mathcal{O}_N(C,A)\vpmb_j}_{=0}\epmb_i^\top = 0.
	\end{equation*}
	This means that, for all $\ell\in\{\mathcal{R}M+1,\ldots,NM \}$ and $s\in[0,T]$,
	$Ce^{(T-s)A}B^\mathcal{O}_\ell=0$,
	and thus
	\begin{equation*}
	\int_{0}^{T}Ce^{(T-s)A}B^\mathcal{O}_\ell\eps(s)ds=0,
	\end{equation*}
	for any control function $\eps \in E_{ad}$.
	Now, recalling the definition of $J(\al)$ from the proof of Lemma \ref{lem: main problem terms of W},
	we write the least-squares problem \eqref{eq: main problem} as
	\begin{align*}
	J(\al) &=\sum_{m=1}^{NM}\norm{\sum_{j=1}^{NM}(\al_j^\star-\al_j)\int_{0}^{T}Ce^{(T-s)A}B^\mathcal{O}_j\eps^m(s)ds}_2^2\\
	&=\sum_{m=1}^{NM}\norm{\sum_{j=1}^{\mathcal{R}M}(\al_j^\star-\al_j)\int_{0}^{T}Ce^{(T-s)A}B^\mathcal{O}_j\eps^m(s)ds}_2^2,
	\end{align*}
	which is our claim.
\end{proof}

Lemma \ref{lem:mainproblemlesscoefficients} implies that the coefficients 
$\al_{\mathcal{R}M+1},\dots,\al_{MN}$ do not affect the cost function to be minimized.
Therefore, any vector $\al \in \mathbb{R}^{MN}$ of the form
$$\al = [ \al_1^\star , \cdots , \al_{\mathcal{R}M}^\star, \gamma_{\mathcal{R}M+1}, \cdots, \gamma_{MN} ]^\top$$
is a global solution to \eqref{eq: main problem}, for any $\gamma_j \in \mathbb{R}$, $j=\mathcal{R}M+1,\dots,MN$.
This means that, one uses really only the first $\mathcal{R}M$ elements of the basis.
In fact, as we are going to show in Lemma \ref{prop: discr positivity 2} and Theorem \ref{thm: main conv special basis},
only their corresponding coefficients can be reconstructed, while no information can be obtained for the remaining ones. 
It is therefore natural, for ${\rm rank} \, \mathcal{O}_N(C,A) = \mathcal{R}<N$, to use the GR algorithm
with only the first $\mathcal{R}M$ basis elements $B^\mathcal{O}_1,\dots,B^\mathcal{O}_{\mathcal{R}M}$.
In this case, the proof of convergence for the GR algorithm is analogous to what we have done to obtain
Theorem \ref{thm: overall main conv}. 
We first prove a version of Lemma \ref{prop: discr positivity} adapted to non-fully observable systems.

\begin{lemma}[Discriminatory-step problem for non-fully observable systems]\label{prop: discr positivity 2}
	Assume that ${\rm rank} \, \mathcal{O}_N(C,A) = \mathcal{R}<N$ and that the GR algorithm is run until the $k$-th
	iteration, with $k<\mathcal{R}M$, using the linearly independent matrices $B^\mathcal{O}_1,\dots,B^\mathcal{O}_{\mathcal{R}M}$
	defined in \eqref{eq:our_basis}. Let $\widehat{W}^k_{[1:k,1:k]}$ be positive definite, 
	and let $\al^k$ be the solution to the fitting-step problem \eqref{eq: fitting step}. 
	Then any solution $\eps^{k+1}$ of the discriminatory-step problem \eqref{eq: discriminatory step} 
	satisfies for $k=1,\ldots \mathcal{R}M-1$
	\begin{equation*}
	\scp{\vpmb, W_{[1:k+1,1:k+1]}(\eps^{k+1})\vpmb}
	=\norm{\int_{0}^{T}Ce^{(T-s)A}\Big(B^\mathcal{O}_{k+1}-\sum_{j=1}^{k}\al^k_j B^\mathcal{O}_j \Big)\eps^{k+1}(s)ds}_2^2>0,
	\end{equation*}
	where $\vpmb:=[(\al^k)^\top,\;-1]^\top$, for $k=0,1,\dots,K-1$.
\end{lemma}

\begin{proof}
Notice that, since the matrices $B^\mathcal{O}_1,\dots,B^\mathcal{O}_{\mathcal{R}M}$ 
are linearly independent and
defined as in \eqref{eq:our_basis}, we have that 
$\mathcal{O}_N(C,A)\Big(B^\mathcal{O}_{k+1}- \sum_{j=1}^{k}\al^k_j B^\mathcal{O}_j \Big)\neq 0$.

With this observation, the result can be proved exactly as Lemma \ref{prop: discr positivity}.
\end{proof}

Using Lemma \ref{prop: discr positivity 2}, we can
prove convergence for the GR Algorithm \ref{algo: main} in case the matrices 
$B_1,\dots,B_{\mathcal{R}M}$ defined in \eqref{eq:our_basis} are used.

\begin{theorem}[Convergence of the GR alg. for non-fully observable systems]\label{thm: main conv special basis}
	Let $(\eps^m)_{m=1}^{\mathcal{R}M}\subset E_{ad}$ be a family of controls generated by the 
	GR Algorithm \ref{algo: main} with $K\leq\mathcal{R}M$ and using the matrices $B^\mathcal{O}_1,\dots,B^\mathcal{O}_{\mathcal{R}M}$ 
	defined in \eqref{eq:our_basis}. Then the least-squares problem 
	\begin{equation}\label{eq: main problem improved}
	\min_{\al\in\mathbb{R}^{\mathcal{R}M}} \sum_{m=1}^{\mathcal{R}M}\norm{C\vp_T(\Bstar,\eps^m)-C\vp_T(B_{\mathcal{R}}(\al),\eps^m) }_2^2
	\end{equation}
	is uniquely solvable with $\al_j=\al_j^\star$ for $j=1,\ldots,\mathcal{R}M$.
\end{theorem}

\begin{proof}
The proof is the same as that of Theorem \ref{thm: overall main conv}, 
where one should use Lemma \ref{prop: discr positivity 2} instead of Lemma \ref{prop: discr positivity}.
\end{proof}

Theorem \ref{thm: main conv special basis} allows us to prove the next corollary,
which characterizes the result of the GR algorithm when more than $\mathcal{R}M$
basis elements of \eqref{eq:our_basis} are used.

\begin{corollary}[More on the convergence for non-fully observable systems]\label{coro:booo}
	Let $(\eps^m)_{m=1}^K\subset E_{ad}$, with $K>\mathcal{R}M$, 
	be a family of controls generated by the GR Algorithm \ref{algo: main}
    using the matrices $B^\mathcal{O}_1,\dots,B^\mathcal{O}_{K}$ defined in \eqref{eq:our_basis}. 
    Then the set of all global minimum points for the
    least-squares problem 
	\begin{equation*}
	\min_{\al\in\mathbb{R}^{K}} \sum_{m=1}^{K}\norm{C\vp_T(\Bstar,\eps^m)-C\vp_T(B_{\mathcal{R}}(\al),\eps^m) }_2^2
	\end{equation*}
	is given by
	$$\{ \al \in \mathbb{R}^K \, : \, \al_j=\al_j^\star, \, j=1,\ldots,\mathcal{R}M \}.$$
\end{corollary}

\begin{proof}
Theorem \ref{thm: main conv special basis} (and
Theorem \ref{thm: overall main conv}) and its proof allow us to obtain that,
using the first $\mathcal{R}M$ controls generated by the GR algorithm, the matrix
$\widehat{W}^{\mathcal{R}M} \in \mathbb{R}^{K \times K}$ has a positive definite
upper-left submatrix $\widehat{W}^{\mathcal{R}M}_{[1:\mathcal{R}M,1:\mathcal{R}M]}$
and all the other entries $[\widehat{W}^{\mathcal{R}M}]_{\ell,j}$ are zero.
Indeed, recalling the vectors $\gam_k(\eps^m)$ defined 
in \eqref{eq: gamma_l,j}, for any $B^\mathcal{O}_k$ with $k\geq \mathcal{R}M+1$, we have
that $\mathcal{O}_N(C,A)B^\mathcal{O}_k=0$ and thus
\begin{equation*}
\gam_k(\eps^m) = \int_{0}^{T}Ce^{(T-s)A}B^\mathcal{O}_k\eps^m(s)ds = 0,
\end{equation*}
for any $T>0$ and any $m=1,\dots,\mathcal{R}M$. Similarly, the matrices
$W(\eps^m)$ for $m > \mathcal{R}M$ have the same structure, namely that their only nonzero
components can be the upper-left submatrices $[W(\eps^m)]_{[1:\mathcal{R}M,1:\mathcal{R}M]}$.
Therefore, the matrix $\widehat{W}=\widehat{W}^K$ has a positive definite
upper-left submatrix $\widehat{W}_{[1:\mathcal{R}M,1:\mathcal{R}M]}$, while all its
other entries are zero. Therefore, the result follows by Lemma \ref{lem: main problem terms of W}.
\end{proof}

\begin{remark}[More about the kernel of $\mathcal{O}_N(C,A)$ and identifiability]\label{rem:after_the_improved_convergence_analysis1}
Corollary \ref{coro:booo} guarantees that, if the basis $(B^\mathcal{O}_j)_{j=1}^K$ is used with $K>\mathcal{R}M$, 
then one can reconstruct exactly $\mathcal{R}M$
coefficients, while nothing can be said about the coefficients $\alpha_j$
for $j>\mathcal{R}M$. This is due to the structure of the matrix $\widehat{W}^{\mathcal{R}M}$,
which has a positive definite submatrix $\widehat{W}^{\mathcal{R}M}_{[1:\mathcal{R}M,1:\mathcal{R}M]}$
and is zero elsewhere (as discussed in the proof of Corollary \ref{coro:booo}).
\end{remark}	

\begin{remark}[a priori error estimate]\label{rem:after_the_improved_convergence_analysis2}	
Let $\al^{approx}$ be the solution to \eqref{eq: main problem improved}. Then we get the a priori error estimate
\begin{equation*}
B^{\star}-B_{\mathcal{R}}(\al^{approx})=\sum_{j=\mathcal{R}M+1}^{NM}\al^{\star}_j B^\mathcal{O}_j.
\end{equation*}
\end{remark}	

\begin{remark}[Min-max problem]\label{rem:after_the_improved_convergence_analysis3}			
Following the same arguments of the proof of Lemma \ref{lem:mainproblemlesscoefficients},
one can show that the min-max problem \eqref{eq:minmaxlincoeff} is equivalent to
\begin{equation}\label{eq:minmaxcoeffimproved}
\min_{\al\in\mathbb{R}^{\mathcal{R}M}}\max_{\eps\in E_{ad}} \norm{C\vp_T(\Bstar,\eps)-C\vp_T(B(\al),\eps)}_2^2.
\end{equation}
Analogously to Remark \ref{rem:minmaxequivmainconv}, we can conclude that, 
using the matrices $B^\mathcal{O}_1,\dots,B^\mathcal{O}_{\mathcal{R}M}$ defined in \eqref{eq:our_basis},
problem \eqref{eq:minmaxcoeffimproved} is uniquely solvable with $\al_j=\alstar_j$ for $j=1,\ldots,\mathcal{R}M$.
\end{remark}

The results proved so far for a non-fully observable system are obtained for the special
basis $(B_j)_{j=1}^{MN}$ constructed in \eqref{eq:our_basis}. However, it is natural to ask:
\begin{itemize}\itemsep0em
\item Does it exist another basis that permits to reconstruct more than $\mathcal{R}M$ coefficients?
\item Can one always reconstruct at least $\mathcal{R}M$ coefficients for any arbitrarily chosen basis?
\end{itemize}

The answer to both questions are negative. 
The first one is given by Theorem \ref{prop:no_basis_recover_more_coefficients}.

\begin{theorem}[Maximal number of identifiable elements]\label{prop:no_basis_recover_more_coefficients}
Let the observability matrix $\mathcal{O}_N(C,A)$ be such that
${\rm rank} \, \mathcal{O}_N(C,A)=\mathcal{R}<N$.
There exists no basis of $\mathbb{R}^{N \times M}$ for which
one can exactly recover more than $\mathcal{R}M$ coefficients.
\end{theorem}

\begin{proof}
	Consider the basis $\mathcal{B}=\{B^\mathcal{O}_k\}_{k=1}^{NM}\subset\mathbb{R}^{N\times M}$ 
	constructed as in \eqref{eq:our_basis}	and another arbitrarily chosen basis 
	$\widehat{\mathcal{B}}=\{\widehat{B}_k\}_{k=1}^{NM}\subset\mathbb{R}^{N\times M}$.
	Any element $\widehat{B}\in\widehat{\mathcal{B}}$ can be written as a linear combination of 
	the elements of $\mathcal{B}$, that is $\widehat{B}=\sum_{j=1}^{NM}\lambda_jB^\mathcal{O}_j$,
	for appropriate $\lambda_j\in\mathbb{R}$, $j=1,\dots,MN$. 
	Multiplying $\widehat{B}$ with $\mathcal{O}_N(C,A)$, we get
	\begin{align*}
	\mathcal{O}_N(C,A)\widehat{B}&=\mathcal{O}_N(C,A)\bigg[\sum_{j=1}^{NM}\lambda_jB^\mathcal{O}_j\bigg]
	=\sum_{j=1}^{NM}\lambda_j\mathcal{O}_N(C,A)B^\mathcal{O}_j
	=\sum_{j=1}^{\mathcal{R}M}\lambda_j\mathcal{O}_N(C,A)B^\mathcal{O}_j,
	\end{align*}
	where we used that $\mathcal{O}_N(C,A)B^\mathcal{O}_j=0$, for $j\in\{\mathcal{R}+1,\ldots,N \}$,
	to obtain the last equality. Now define the set $\mathcal{D}=\{D_k\}_{k=1}^{NM}$ as
	\begin{equation*}
	D_k:=\mathcal{O}_N(C,A)\widehat{B}_k,\quad k=1,\ldots,NM.
	\end{equation*}
	Hence, we can conclude that at most $\mathcal{R}M$ elements of $\mathcal{D}$ are linearly independent.
	Recalling the proof of Lemma \ref{prop: discr positivity} and Remark \ref{rem:after_the_improved_convergence_analysis1},
	this means that for $NM-\mathcal{R}M$ elements of $\widehat{\mathcal{B}}$ 
	there exists a linear combination of the other $\mathcal{R}M$ elements, 
	such that the observation at final time $T$ is identical for any control $\eps$.
	Therefore one can reconstruct at most $\mathcal{R}M$ coefficients for the basis $\widehat{\mathcal{B}}$.
\end{proof}

Let us now explain why the answer to the second question is also negative.
To do so, we provide the following examples, which show that
a wrong choice of a basis leads to inconclusive results.

\begin{example}[Wrong bases lead to inconclusive results]\label{example:bad}
	Consider a simple system with
	\begin{equation*}
	A=\begin{bmatrix}
	1&0\\0&1
	\end{bmatrix}, \;
	B^\star = \begin{bmatrix}
	1&1\\1&1
	\end{bmatrix}, \;
	C=\begin{bmatrix}
	1&0\\0&0
	\end{bmatrix},
	\end{equation*}	
	and the basis of $\mathbb{R}^{2 \times 2}$
	\begin{equation*}
	\widehat{B}_1 = \begin{bmatrix}
	1&0\\0&0
	\end{bmatrix},
	\widehat{B}_2 = \begin{bmatrix}
	1&0\\1&0
	\end{bmatrix},
	\widehat{B}_3 = \begin{bmatrix}
	0&1\\0&0
	\end{bmatrix},
	\widehat{B}_4 = \begin{bmatrix}
	0&1\\0&1
	\end{bmatrix}.
	\end{equation*}
	Notice that in this case the observability condition does not hold,
	since one can compute that $\mathcal{R}={\rm rank}\,\mathcal{O}_N(C,A)
	={\rm rank}\begin{small}\begin{bmatrix}
	1 & 0 & 1 & 0 \\
	0 & 0 & 0 & 0 \\
	\end{bmatrix}^\top\end{small}=1$.
	Clearly we have that
	\begin{equation*}
	B^\star = 0\cdot B_1+1\cdot B_2+0\cdot B_3+1\cdot B_4, \; \text{ (hence $\al^\star = [ 0 \, 1 \, 0 \, 1]^\top$)}.
	\end{equation*}	
	We can now compute for an arbitrarily chosen control $\eps \in E_{ad}$ that
	\begin{align*}
	C\vp_T(B^\star,\eps)-&C\vp_T(B(\al),\eps)=C\int_{0}^{T}e^{(T-s)A}B^\star\eps(s)ds-C\int_{0}^{T}e^{(T-s)A}B(\al)\eps(s)ds\\
	&=\int_{0}^{T}Ce^{(T-s)A}\Big(\begin{bmatrix}
	1&1\\1&1
	\end{bmatrix}-\begin{bmatrix}
	\al_1+\al_2&\al_3+\al_4\\\al_2&\al_4
	\end{bmatrix}\Big)\eps(s)ds\\
	&=\int_{0}^{T}\begin{bmatrix}
	1&0\\0&0
	\end{bmatrix}\begin{bmatrix}
	e^{T-s}&0\\0&e^{T-s}
	\end{bmatrix}\begin{bmatrix}
	1-(\al_1+\al_2)&1-(\al_3+\al_4)\\1-\al_2&1-\al_4
	\end{bmatrix}\eps(s)ds\\
	&=\int_{0}^{T}
	\begin{bmatrix}
	e^{T-s}(1-(\al_1+\al_2))&e^{T-s}(1-(\al_3+\al_4))\\
	0&0
	\end{bmatrix}\eps(s)ds,
	\end{align*}
	which is zero for any $\al=[\al_1\; \al_2\; \al_3\; \al_4]^\top \in\mathbb{R}^4$ with $\al_1+\al_2=1$ and $\al_3+\al_4=1$ (for any control $\eps$).
	This means that any $\al=[\al_1\; \al_2\; \al_3\; \al_4]$ with $\al_1+\al_2=1$ and $\al_3+\al_4=1$ solves the least-squares
	problem \eqref{eq: main problem}, independently on the control functions $\eps_1,\dots,\eps_4$.
	Since the online least-square problem has then infinitely many solutions,\footnote{Notice that these
	solutions are also solution to the min-max problem \eqref{eq:minmaxlinear}-\eqref{eq:minmaxlincoeff}.} 
	one cannot conclude anything about the quality of a computed solution,
	which has the form
	\begin{equation*}
	\widehat{B}^{approx}=\begin{bmatrix}
	1&1\\\al_2&\al_4
	\end{bmatrix},
	\end{equation*}
	leading to the error
	\begin{equation*}
	\norm{B^{\star}-B_{\mathcal{R}}(\al^{approx})}_F^2=(1-\al_2)^2+(1-\al_4)^2,
	\end{equation*}
	which can be arbitrarily large (here $\norm{\,\cdot\,}_F$ denotes the Frobenius norm).
	Even if one would by chance guess the right coefficients (in this case $\al_2=1, \al_4=1$)
	there would be no way to verify it, since their effect is not observable.
	Notice also that, even if the entries $\widehat{B}^{approx}_{1,1}$ and $\widehat{B}^{approx}_{1,2}$ are correct,
	it is not possible to certify this or to associate these correct entries to some precise elements of the chosen basis.
	This example shows that for an arbitrarily chosen basis, one can not conclude anything about the quality of the computed 
	coefficients or	the difference between $B(\al)$ and $B^{\star}$. 
\end{example}

\begin{example}[Good bases lead to certified results]
Consider the same system of Example \ref{example:bad}, but now let us
use the SVD of the observability matrix,
	\begin{equation*}
	\mathcal{O}_2(C,A)=\begin{bmatrix}
	1&0\\0&0\\1&0\\0&0
	\end{bmatrix} = 
	\begin{bmatrix}
	\frac{\sqrt{2}}{2}&0&-\frac{\sqrt{2}}{2}&0\\
	0&1&0&0\\
	\frac{\sqrt{2}}{2}&0&\frac{\sqrt{2}}{2}&0\\
	0&0&0&1
	\end{bmatrix}
	\begin{bmatrix}
	\sqrt{2}&0\\0&0\\0&0\\0&0
	\end{bmatrix}
	\begin{bmatrix}
	1&0\\0&1
	\end{bmatrix}=U\Sigma V^\top,
	\end{equation*}
	which gives
	\begin{equation*}
	\vpmb_1 = \begin{bmatrix}
	1\\0
	\end{bmatrix}\notin \ker \obsmat,\quad\vpmb_2 = \begin{bmatrix}
	0\\1
	\end{bmatrix}\in \ker \obsmat,
	\end{equation*}
	leading to the basis (constructed as in \eqref{eq:our_basis})
	\begin{equation*}
	B_1 = \begin{bmatrix}
	1&0\\0&0
	\end{bmatrix},
	B_2 = \begin{bmatrix}
	0&1\\0&0
	\end{bmatrix},
	B_3 = \begin{bmatrix}
	0&0\\1&0
	\end{bmatrix},
	B_4 = \begin{bmatrix}
	0&0\\0&1
	\end{bmatrix}.
	\end{equation*}
	In this case, we have $\al^\star=[1\;1\;1\;1]^\top$. 
	Since the GR algorithm considers only the first two basis elements,
	one gets the final result
	\begin{equation*}
	B^{approx}=\begin{bmatrix}
	1&1\\0&0
	\end{bmatrix}.
	\end{equation*}
	Similarly to Example \ref{example:bad}, the two entries 
	$\widehat{B}^{approx}_{1,1}$ and $\widehat{B}^{approx}_{1,2}$ are correct,
	but now this is guaranteed by Theorem \ref{thm: main conv special basis}. 
	Therefore, in this case, the results obtained are accompanied by precise 
	information on their correctness (guaranteed by theoretical results).
\end{example}

These examples show clearly that without an a priori knowledge about the observability of the system 
(and hence about the ``quality'' of the basis), the GR algorithm leads to inconclusive results.
Even though we have presented in this section a way to construct 
a basis which permits a precise analysis of the obtained results,
this is generally not possible for nonlinear problems, like the Hamiltonian reconstruction 
problem described in Section \ref{sec:bilinear_setting}.
Is it then possible to modify the GR algorithm in order to distinguish automatically between
``good'' and ``bad'' elements of a given set of matrices?
The answer is given in Section \ref{sec:impro}, where we first introduce an improved GR algorithm 
for linear-quadratic problems and then extend it to nonlinear problems.

%

\section{Improvements of the algorithm}\label{sec:impro}

The previous section ended with two examples showing clearly that a wrong choice of the
basis elements and their ordering can lead to inconclusive results.
Even though this issue can be avoided for linear problems by using the observability matrix
(and constructing a basis as in \eqref{eq:our_basis}), this strategy does generally not apply to
nonlinear problems. For this reason, we introduce an optimized GR (OGR) algorithm,
in which the basis elements are selected during the iterations (in a greedy fashion)
as the ones that maximize the discrimination functions. In particular, we introduce
in Section \ref{sec:impro_greedy} the OGR algorithm for linear-quadratic problems and show by numerical 
experiments that this leads to an automatic appropriate selection of the basis elements, 
even though the observability matrix is not considered at all.
Once the new algorithm is introduced for linear systems, it is then natural to extend it to
nonlinear problems. We consider this extension in Section \ref{sec:impro_com} for 
Hamiltonian reconstruction problems and show the efficiency
of our new OGR algorithm by direct numerical experiments.

\subsection{Optimized greedy reconstruction for linear-quadratic problems}\label{sec:impro_greedy}

Consider an arbitrary set of linearly independent matrices $(B_j)_{j=1}^K \subset \mathbb{R}^{N \times M}$. 
We wish to modify the GR Algorithm \ref{algo: main} in order to choose at every iteration
one element $B_j$ which leads to a control function capable of improving the rank of the matrix
$\widehat{W}^k_{[1:k+1,1:k+1]}$. The idea is to replace the sweeping process of the GR 
Algorithm \ref{algo: main} with a more robust and parallel testing of all the matrices.
At each iteration, the element associated with the maximal discriminating value is chosen 
and removed from the set $(B_j)_{j=1}^K$, while the corresponding control function is added to the set 
of already computed control functions. 
Therefore, the dimension of the set $(B_j)_{j=1}^K$ reduces by one at each iteration and 
the algorithm is stopped if either all the $K$ matrices where chosen or as soon none of the remaining ones
can be discriminated by the others. This idea lead to the OGR Algorithm \ref{alg: greedy improvement}.

\begin{algorithm}[t]
	\caption{Optimized Greedy Reconstruction Algorithm (linear-quadratic case)}
	\begin{small}
	\begin{algorithmic}[1]\label{alg: greedy improvement}
		\REQUIRE A set of $K$ linearly independent matrices $\mathcal{B}=(B_1,\ldots,B_K)$ 
		and a tolerance tol$>0$.
		\STATE Solve the initialization problem
		$$\max_{\ell\in\{1,\ldots,K \}}\max_{\eps \in E_{ad}} \norm{C\vp_T(B_\ell,\eps)-C\vp_T(0,0)}_2^2$$
		which gives the field $\eps^1$ and the index $\ell_1$.
		\IF {$\norm{C\vp_T(B_{\ell_1},\eps^1)-C\vp_T(0,0)}_2^2< {\rm tol}$}
		\STATE {\bf stop}.
		\ENDIF
		\STATE Switch $B_1$ and $B_{\ell_1}$ in $\mathcal{B}$ and set $k=1$.
		\WHILE{ $k\leq K-1$ }
		\FOR{$\ell=k+1,\ldots,K$}
		\STATE \underline{Fitting step}: Find $(\al^{\ell}_j)_{j=1,\dots,k}$ that solve the problem
		\begin{equation}\label{eq: greedy fitting step}
		\min\limits_{\al \in\mathbb{R}^k}\sum_{m=1}^{k}\norm{C\vp_T(B_{\ell},\eps^m)-C\vp_T(B(\al),\eps^m) }_2^2.
		\end{equation}
		\ENDFOR
		\STATE \underline{Extended discriminatory step}: Find $\eps^{k+1}$ and $\ell_{k+1}$ that solve the problem
		\begin{equation}\label{eq: greedy discriminatory step}
		\max_{\ell\in\{k+1,\ldots,K \}}\max\limits_{\eps\in E_{ad}}\norm{C\vp_T(B_\ell,\eps)-C\vp_T(B(\al^\ell),\eps) }_2^2.
		\end{equation}
		\IF {$\norm{C\vp_T(B_{\ell_{k+1}},\eps^{k+1})-C\vp_T(B(\al^{\ell_{k}}),\eps^{k+1})}_2^2< {\rm tol}$}
		\STATE {\bf stop} and return the selected $(B_j)_{j=1}^k$ and the computed $(\eps^m)_{m=1}^k$.
		\ENDIF
		\STATE Switch $B_{k+1}$ and $B_{\ell_{k+1}}$ in $\mathcal{B}$ and update $k \leftarrow k+1$.
		\ENDWHILE
	\end{algorithmic}
	\end{small}
\end{algorithm}

In this algorithm, we clearly extended the greedy character of the original GR algorithm 
to the choice of the next basis element. At each iteration, we consider all remaining basis 
elements as the potential next one. We select the one which yields the largest function 
value in the respective discrimination (maximization) step. 
In other words, one computes the basis element for which one can split the observation the most 
from all previous basis elements. 
It is important to remark that, at each iteration one solves several fitting-step problems
and several discriminatory-step problems. However, their solving can be performed 
in parallel, since the single problems are independent one from another.

Notice that a selected element $B_{k+1}$ will not be linearly dependent on previously
chosen elements (after multiplication with the observability matrix). This is proven in the next theorem, which also motivates the stopping criterion used in the 
steps 2-4 and 11-13 of the algorithm.

\begin{theorem}[Linearly independence of selected basis elements]\label{thm:lin_ind}
Assume that the OGR Algorithm \ref{alg: greedy improvement} selected already $k$
linearly independent matrices $B_j$, $j=1,\dots,k$. 
At iteration $k+1$, the new selected matrix $B_{k+1}$ is such that
$\mathcal{O}_N(C,A)B_{k+1}$ is linearly independent
from the matrices $\mathcal{O}_N(C,A)B_j$, $j=1,\dots,k$, if and only if 
$$\norm{C\vp_T(B_{\ell_{k+1}},\eps^{k+1})-C\vp_T(B(\al^{\ell_{k}}),\eps^{k+1})}_2^2>0.$$
\end{theorem}

\begin{proof}
If the matrix $\mathcal{O}_N(C,A)B_{k+1}$ is linearly independent
from the other matrices $\mathcal{O}_N(C,A)B_j$, $j=1,\dots,k$, then one can show
as in the proof of Lemma \ref{prop: discr positivity 2} that 
$$\norm{C\vp_T(B_{\ell_{k+1}},\eps^{k+1})-C\vp_T(B(\al^{\ell_{k}}),\eps^{k+1})}_2^2>0.$$

Now, we prove the other implication by contraposition.
Assume that there exists a vector $\al\in\mathbb{R}^k$ such that
$\obsmat (B_{k+1}-\sum_{j=1}^k\al_jB_j)=0$ holds. This vector $\al$ is a solution 
of the fitting step problem with cost-function value equal to zero. 
However, the corresponding cost function of the discriminatory-step problem
\eqref{eq: greedy discriminatory step}
results to be zero for any control function $\eps$. 
The result follows by contraposition.
\end{proof}

Theorem \ref{thm:lin_ind} shows exactly that 
the OGR algorithm manages to identify among the elements of the given set $(B_j)_{j=1}^K$ 
the ones that do not lie in the kernel of $\mathcal{O}_N(C,A)$.
For instance, let us consider again the system of Example \ref{example:bad}, 
for which we have shown that the GR algorithm leads to inconclusive results. 
If we use instead the OGR Algorithm \ref{alg: greedy improvement},
this performs two iterations and selects only two basis elements,
one among $\widehat{B}_1$ and $\widehat{B}_2$ and the other among
$\widehat{B}_3$ and $\widehat{B}_4$. This can be shown by performing
calculations similar to the ones of Example \ref{example:bad}.
In particular, in the initialization step the four matrices produce the 
same cost function value. Hence, any of them can be selected by the algorithm.
Assume that the element $\widehat{B}_1$ is picked (hence $\ell_1=1$)
and consider the first iteration of the algorithm ($k=1$).
At the fitting step the algorithm computes a coefficient $\al^2_1=1$ for $\widehat{B}_2$,
and some coefficients $\al^3_1$ and $\al^4_1$ corresponding to $\widehat{B}_3$
and $\widehat{B}_4$. Now, $\al^2_1=1$ leads to a cost function
of the discriminatory step which is zero for any control functions, while
for $\al^3_1$ and $\al^4_1$ there exist a control function leading to a non-zero value
of the discriminatory cost.
Therefore, the algorithm selects either $\widehat{B}_3$ or $\widehat{B}_4$. 
Let us assume that $\widehat{B}_4$ is picked ($\ell_2=4$) and hence the two 
elements $\widehat{B}_2$ and $\widehat{B}_4$ are switched.
In the fitting step of the second iteration ($k=2$), the algorithm
computes $\al^3=[\, 0 \, , \, 1 \,]^\top$ and $\al^4=[\, 1 \, , \, 0 \,]^\top$.
Both of these two vectors lead to a discriminatory cost that is
zero for any control. Hence, since the discriminatory step does
not find any positive function value, the algorithm stops and returns
$\widehat{B}_{\ell_1}=\widehat{B}_1$ and $\widehat{B}_{\ell_2}=\widehat{B}_4$
and the corresponding controls. 
If one uses the two selected basis elements and the corresponding control functions
in the online phase, then one obtains the result $\al = [ \, 1 \, , \, 1 \, ]^\top$,
which is not the exact solution shown in Example \ref{example:bad}.
This is due to the non-full observability of the system, which
implies that $\mathcal{O}_N(C,A)\widehat{B}_1 = \mathcal{O}_N(C,A)\widehat{B}_2$
and $\mathcal{O}_N(C,A)\widehat{B}_3 = \mathcal{O}_N(C,A)\widehat{B}_4$. This means
that the observations generated by the elements $\widehat{B}_1$ and $\widehat{B}_3$ 
cannot be distinguished by the ones created by $\widehat{B}_2$ and $\widehat{B}_4$. 
The non-full observability of the system cannot be
overcome by any numerical strategy. The OGR algorithm can nevertheless identify
automatically all the observable degrees of freedom of the considered system.

Let us now demonstrate the efficiency of our new OGR algorithm by direct numerical experiments.
We consider an experiment with two randomly chosen $N \times N$ full-rank real
matrices $A$ and $C$ with $N=10$. 
The unknown $B^\star$ is a randomly chosen real $N \times N$  matrix.
In this case the system is fully observable, nevertheless we construct the basis elements
to be used in the GR and OGR algorithm as in \eqref{eq:our_basis} (by an SVD of the observability matrix), 
but we order the elements randomly. We then run the GR Algorithm \ref{algo: main}
and compute the rank of the matrix $\widehat{W}^k$ at every iteration $k$.
This leads to the results shown in Figure \ref{fig:rank} by the blue curve.
\begin{figure}
\centering
\includegraphics[scale=0.38]{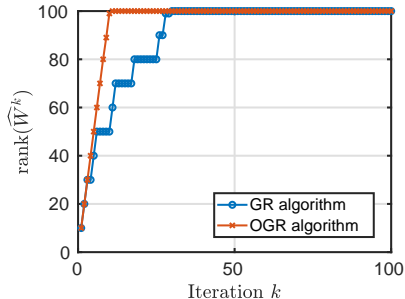}
\caption{Rank of the matrix $\widehat{W}^k$ corresponding to the GR algorithm (blue curve)
and OGR algorithm (red curve) for a fully observable system.
Both algorithms make use of a basis constructed as in \eqref{eq:our_basis}.}
\label{fig:rank}
\end{figure}
The rank increases monotonically during the iterations and becomes full after about 30 iterations.
However, the curve is not strictly monotonically increasing since the rank does not increase 
at each iteration. If we repeat the same experiment (with the same matrices) using
the OGR Algorithm \ref{alg: greedy improvement}, we obtain the red curve in Figure 
\ref{fig:rank}. This curve is strictly monotonically increasing in the first part
and becomes constant only once the rank has become full. In particular, at each iteration
the rank increases by 10 and the OGR algorithm could be in principle stopped much earlier
than the original GR algorithm, and much less control functions (hence laboratory experiments)
are needed to fully reconstruct the unknown operator $B^\star$.
This experiment clearly shows the high potential of the OGR
algorithm, which is capable to choose among the elements $B_1,\dots,B_K$ in an optimized fashion.

Let us conclude this section with two important observations.
First, the improvement proposed in Algorithm \ref{alg: greedy improvement} allows one
to even enrich the set $(B_j)_{j=1}^{K}$ used as input in Algorithm \ref{alg: greedy improvement}
with other new elements that can be linearly 
dependent on $B_1,\ldots,B_K$. In this case, if we denote by  $(B_j)_{j=1}^{\widetilde{K}}$,
for $\widetilde{K}>K$, the enriched set, then Theorem \ref{thm:lin_ind} guarantees that 
the OGR algorithm will automatically pick some elements of the enriched set
$(B_j)_{j=1}^{\widetilde{K}}$, such that  $\mathcal{O}_N(C,A)B_j$ are linearly independent for
all selected $B_j$. Hence, the corresponding discriminatory
cost-function values will be strictly positive.
Second, the OGR Algorithm can be extended to more general nonlinear reconstruction problems, 
and we propose in Section \ref{sec:impro_com} an efficient extension for the Hamiltonian reconstruction
problem described in Section \ref{sec:bilinear_setting} at the beginning of this paper.

\subsection{Optimized greedy reconstruction for non-linear problems}\label{sec:impro_com}

The extension of the OGR Algorithm \ref{alg: greedy improvement} to the nonlinear
Hamiltonian reconstruction problem of Section \ref{sec:bilinear_setting} is formally
rather straightforward and given by Algorithm \ref{alg: greedy improvement Hamilt}.
\begin{algorithm}[t]
	\caption{Optimized Greedy Reconstruction Algorithm (Hamiltonian case)}
	\begin{small}
	\begin{algorithmic}[1]\label{alg: greedy improvement Hamilt}
		\REQUIRE A set of $K$ matrices $\mathcal{B}_\mu=(\mu_\ell)_{\ell=1,\dots,K}$
		and a tolerance tol$>0$.
		\STATE Solve the initialization problem
		\begin{equation}\label{eq:bilinear_initialization_regularized}
		 \max_{n\in\{1,\ldots,K \}}\max_{\epsilon \in L^2} \vert
		 \varphi(\mu_n,\epsilon)-\varphi(0,0)
		 \vert^2.   
		\end{equation}
		which gives the field $\epsilon^1$ and the index $\ell_1$.
		\IF {$\vert
		 \varphi(\mu_{\ell_1},\epsilon^1)-\varphi(0,0)
		 \vert^2< {\rm tol}$}
		\STATE {\bf stop}.
		\ENDIF
		\STATE Switch $\mu_1$ and $\mu_{\ell_1}$ in $\mathcal{B}_\mu$ and set $k=1$.
		\WHILE{ $k\leq K-1$ }
		\FOR{$\ell=k+1,\ldots,K$}
		\STATE \underline{Fitting step}: Find $(\al^{\ell}_j)_{j=1,\dots,k}$ that solve the problem
		\begin{equation}\label{eq: greedy fitting step Ham}
		\min\limits_{\al \in\mathbb{R}^k}\sum_{m=1}^{k}|\varphi(\mu_{k+1},\epsilon^m)-\varphi(\mu^k(\al),\epsilon^m) |^2.
		\end{equation}
		\ENDFOR
		\STATE \underline{Extended discriminatory step}: Find $\epsilon^{k+1}$ and $\ell_{k+1}$ that solve the problem
		\begin{equation}\label{eq: greedy discriminatory step Ham}
		\max_{\ell\in\{k+1,\ldots,K \}}\max\limits_{\epsilon\in L^2}|\varphi(\mu_{k+1},\epsilon)-\varphi(\mu^k(\al^k),\epsilon) |^2.
		\end{equation}
		\IF {$|\varphi(\mu_{\ell_{k+1}},\epsilon^{k+1})-\varphi(\mu^k(\al^{\ell_k}),\epsilon^{k+1}) |^2< {\rm tol}$}
		\STATE {\bf stop} and return the selected $(\mu_j)_{j=1}^k$ and the computed $(\epsilon^m)_{m=1}^k$.
		\ENDIF
		\STATE Switch $\mu_{k+1}$ and $\mu_{\ell_{k+1}}$ in $\mathcal{B}_\mu$ and update $k \leftarrow k+1$.
		\ENDWHILE
	\end{algorithmic}
	\end{small}
\end{algorithm}
However, a few more computational aspects must be discussed.
First, the maximization problems characterizing the initialization step and the discriminatory 
steps are nonlinear optimal control problems that we solve numerically by the monotonic
scheme discussed in \cite{monschemes}; see also \cite{monotonic_analysis,madaysalomon,LibroQuantum,monschemes} 
and references therein. Second, the fitting step problems are highly nonlinear minimization problems
having generally several local minima. Since not all local minima correspond to an effective
defect (rank deficiency in the linear-quadratic case) to be compensated, every fitting-step problem
is solved multiple times using different randomly chosen initializations. The solution corresponding
to the smallest functional value is then chosen. Each fitting-step problem is solved by
a BFGS descent-direction method. Third, all optimization problems that are solved in the fitting steps
and in the discriminatory steps are independent one from another. Therefore, they can be solved in parallel as in the linear case.

Let us now show the efficiency of the OGR Algorithm \ref{alg: greedy improvement Hamilt} by direct
numerical experiments. We consider the same test case as in \cite{madaysalomon}, where the 
internal Hamiltonian and the unknown (randomly generated) dipole moment matrix are
\begin{equation*}
    H = 10^{-2}
    \begin{footnotesize}
    \begin{bmatrix}
    1&0&0\\
    0&2&0\\
    0&0&4
    \end{bmatrix}
    \end{footnotesize}, \;
    \mu^\star = 
    \begin{footnotesize}
    \begin{bmatrix}
    3.3617  &  3.4347  &  0.8416 \\
    3.4347  &  3.7763  &  4.7552 \\
    0.8416  &  4.7552  &  4.4226 \\
    \end{bmatrix}
    \end{footnotesize}.
\end{equation*}
The final time is $T=4000\pi$.
The states $\psi_0$ and $\psi_1$ are

\begin{equation*}
    \psi_0=\begin{bmatrix}1 & 0 & 0\end{bmatrix}^\top
    ,\quad
    \psi_1=\begin{bmatrix}0 &0 &1\end{bmatrix}^\top.
\end{equation*}

Now, we perform the following experiment. 
Since the unknown $\mu^\star$ is a $3\times 3$ symmetric matrix, 
we choose for the set $\mathcal{B}_\mu$
the following $K=6$ linearly independent canonical matrices

\vspace*{-3mm}
\begin{equation*}
\begin{small}
\begingroup 
\setlength\arraycolsep{2.4pt}
\begin{bmatrix}
1 & 0 & 0 \\
0 & 0 & 0 \\
0 & 0 & 0 \\
\end{bmatrix},
\begin{bmatrix}
0 & 0 & 0 \\
0 & 1 & 0 \\
0 & 0 & 0 \\
\end{bmatrix},
\begin{bmatrix}
0 & 0 & 0 \\
0 & 0 & 0 \\
0 & 0 & 1 \\
\end{bmatrix},
\begin{bmatrix}
0 & 1 & 0 \\
1 & 0 & 0 \\
0 & 0 & 0 \\
\end{bmatrix},
\begin{bmatrix}
0 & 0 & 1 \\
0 & 0 & 0 \\
1 & 0 & 0 \\
\end{bmatrix},
\begin{bmatrix}
0 & 0 & 0 \\
0 & 0 & 1 \\
0 & 1 & 0 \\
\end{bmatrix},
\endgroup
\end{small}
\end{equation*}

\noindent
which form a basis for ${\rm Sym}(3)$,
and compute 6 control functions by the OGR Algorithm \ref{alg: greedy improvement Hamilt}.
Once these functions are obtained, one must reconstruct the unknown true dipole matrix by solving
the online nonlinear least-squares problem \eqref{eq:bilinear_main_problem}.
To do so, we use the standard MATLAB function \texttt{fminunc} (a BFGS descent-direction
minimization algorithm) initialized by a randomly chosen vector.
To test the robustness of the control functions computed by the 
OGR Algorithm \ref{alg: greedy improvement Hamilt}, we consider a six-dimensional hypercube 
centered in the global minimum point $\al^\star$ and given radius $r$, and
repeat the minimization for 1000 initialization vectors randomly chosen in this hypercube. 
We then count the number of times that the optimization algorithm converges to
the global solution $\mu^\star=\mu(\al^\star)$ up to a tolerance of 0.005
(half of the smallest considered radius).
Repeating this experiment for different values of the radius $r$ of the hypercube,
we obtain the results reported in the first row of Table \ref{tab:1}.

\begin{table}[!ht]
\centering
\begin{tabular}{ l|c|c|c|c|c|c } 
 \hline
 Hypercube radius $r$ & 0.01 & 0.10 & 0.25 & 0.50 & 0.75 & 1.00 \\ \hline
 GR (canonical basis) & 23 & 0 & 0 & 0 & 0 & 0 \\ 
 GR (random basis) & 431 & 43 & 15 & 8 & 1 & 0 \\
 OGR (extended random basis) & 939 & 923 & 884 & 735 & 627 & 541 \\
 \hline
\end{tabular}
\caption{Numbers of runs (over 1000) that converged to the true solution $\al^\star$.}
\label{tab:1}
\end{table}

\vspace*{-5mm}
\noindent
These results show clearly the lack of robustness of the controls generated by the GR algorithm:
in only 23 cases over the 1000 runs the minimization converged to the true solution,
and this happens only for the smallest radius $r=0.01$ of the hypercube. 

Next, to test the effect of the chosen basis $\mathcal{B}_\mu$, 
we repeat the same experiment using 6 randomly chosen linearly independent
symmetric matrices $\mu_\ell$, $\ell=1,\dots,6$.
The obtained results of this second test are shown in the second row of Table \ref{tab:1}.
These are clearly better, but still very unsatisfactory.

Finally, we repeat the experiment using the OGR Algorithm \ref{alg: greedy improvement Hamilt}
with a set of 12 matrices, namely the 6 unit basis elements shown above and the 6 linearly
independent random matrices chosen for the second experiment.
We obtain the results shown in the third row of Table \ref{tab:1}. 
These are now very satisfactory. Even though, the number of times that the optimization algorithm 
converged to the true solution decays as the radius $r$ increases, in the worst case for $r=1.00$
more than 500 of runs converged to $\al^\star$. These results show clearly the efficiency
of the new proposed OGR algorithm.

\section{Conclusions}\label{sec:conclusions}

In this work we provided a novel and detailed convergence analysis for
the greedy reconstruction algorithm introduced in \cite{madaysalomon} for Hamiltonian 
reconstruction problems in the field of quantum mechanics. 
The presented convergence analysis has considered linear-quadratic (optimization, least-squares) 
problems and revealed the strong dependence of the performance of the greedy reconstruction algorithm 
on the observability properties of the system and on the ansatz of the basis elements used to reconstruct 
the unknown operator. This allowed us to introduce a precise (and in some sense optimal) choice of the basis elements for the linear case and led to the introduction of an optimized greedy reconstruction algorithm
applicable also to the nonlinear Hamiltonian reconstruction problem.
Numerical experiments demonstrated the efficiency of the new proposed numerical algorithm.

\vspace*{-2mm}
\bibliographystyle{abbrv}
\bibliography{references}

\end{document}